\documentclass[reqno]{amsart}

\usepackage{amssymb}
\usepackage{a4wide}
\usepackage{enumerate}
\usepackage{esint}
\usepackage[colorlinks=true, urlcolor=blue, linkcolor=blue, citecolor=black]{hyperref}
\usepackage{mathrsfs}
\usepackage{titletoc}
\usepackage[usenames,dvipsnames,svgnames,table]{xcolor}
\usepackage[backgroundcolor=white, bordercolor=blue, linecolor=blue]{todonotes}
\usepackage[nameinlink]{cleveref}
\usepackage{mathtools}

\makeatletter
\newcommand{\mylabel}[2]{(#2)\def\@currentlabel{#2}\label{#1}}
\makeatother

\theoremstyle{plain}
\newtheorem{theorem}{Theorem}[section]
\newtheorem{lemma}[theorem]{Lemma}
\newtheorem{corollary}[theorem]{Corollary}
\newtheorem{proposition}[theorem]{Proposition}

\theoremstyle{definition}
\newtheorem{definition}[theorem]{Definition}
\newtheorem{remark}[theorem]{Remark}
\newtheorem*{remark*}{Remark}
\newtheorem{example}[theorem]{Example}

\numberwithin{equation}{section}

\newcommand{\R}{\mathbb{R}}
\newcommand{\Rd}{\mathbb{R}^d}
\renewcommand{\d}{\mathrm{d}} 
\newcommand{\dx} {\, \mathrm{d}x}
\newcommand{\dy} {\, \mathrm{d}y}
\newcommand{\dt} {\, \mathrm{d}t}

\newcommand{\cl} {c_{\ast}}
\newcommand{\cU} {c_U}
\newcommand{\cK} {c_K}
\newcommand{\cH} {c_H}
\newcommand{\cJ} {c_J}

\makeatletter
\newcommand{\setword}[2]{%
  \phantomsection
  #1\def\@currentlabel{\unexpanded{#1}}\label{#2}%
}
\makeatother

\makeatletter
\@namedef{subjclassname@2020}{\textup{2020} Mathematics Subject Classification}
\makeatother

\title[]{Robust near-diagonal Green function estimates}

\author{Moritz Kassmann}
\address{Fakult\"at f\"ur Mathematik, Universit\"at Bielefeld, 33615 Bielefeld, Germany}
\email{moritz.kassmann@uni-bielefeld.de}

\author{Minhyun Kim}
\address{Department of Mathematics, Hanyang University, 04763 Seoul, Republic of Korea}
\email{minhyun@hanyang.ac.kr}

\author{Ki-Ahm Lee}
\address{Department of Mathematical Sciences \& Research Institute of Mathematics, Seoul National University, 08826 Seoul, Republic of Korea}
\email{kiahm@snu.ac.kr}

\thanks{Moritz Kassmann and Minhyun Kim gratefully acknowledge financial support by the German Research Foundation (SFB 1283 - 317210226 resp. GRK 2235 - 282638148). The research of Minhyun Kim is supported by the research fund of Hanyang University (HY-202300000001143). The research of Ki-Ahm Lee is supported by the National Research Foundation of Korea (NRF) grant funded by the Korea government (MSIP): NRF-2021R1A4A1027378.}

\begin{document}

\begin{abstract}
We prove sharp near-diagonal pointwise bounds for the Green function $G_\Omega(x,y)$ for nonlocal operators of fractional order $\alpha \in (0,2)$. The novelty of our results is two-fold: the estimates are robust as $\alpha \to 2-$ and we prove the bounds without making use of the Dirichlet heat kernel $p_\Omega(t;x,y)$. In this way we can cover cases, in which the Green function satisfies isotropic bounds but the heat kernel does not. 
\end{abstract}

\maketitle

\section{Introduction} \label{sec:introduction}

An important object in Potential Analysis is the Green function. Explicit formulas for the Green function of the fractional Laplace operator for $0 < \alpha < 2$ are complicated, even for domains like a ball in the Euclidean space, see \cite{Rie38, BGR61} or more recently \cite{Buc16}. Let $G^{(\alpha)}$ be the Green function for the unit ball in $\R^d$ and assume $d \geq 3$. A careful study of the constants shows for $x_0 \in B$
\begin{align*}
\lim\limits_{\alpha \to 2} \, \lim\limits_{\substack{x \to x_0 \\ y \to x_0}} \frac{G^{(\alpha)} (x,y)}{|x-y|^{\alpha -d}} = \lim\limits_{\alpha \to 2} \, \frac{\Gamma(\frac{d-\alpha}{2})}{2^\alpha \pi^{d/2} \Gamma(\frac{\alpha}{2})} = \frac{1}{d(d-2)|B_1|}.
\end{align*}

As a consequence, one can expect near-diagonal pointwise bounds for the Green function of a given nonlocal operator of fractional order $\alpha \in (0,2)$ with constants, which are independent of $\alpha$ away from zero. Establishing such results is one aim of the current paper. 

There exist several very detailed works on pointwise bounds for the Green function for nonlocal operators of fractional order such as \cite{CKS10a,BGR14}. The authors establish fine near-diagonal and boundary estimates, which are sharp as long as $\alpha \in (0,2)$ is fixed but these results contain constants, which degenerate as $\alpha \to 2-$. The reason for this phenomenon is that the main tool inside the proofs is the formula $G^{(\alpha)} (x,y) = \int_0^\infty p^{(\alpha)}_\Omega(t;x,y) \d t$, where $p^{(\alpha)}_\Omega$ is the corresponding Dirichlet heat kernel. Since the decay property of $p^{(\alpha)}_\Omega$ for large values of $|x-y|$ changes from polynomial to exponential at $\alpha =2$, this approach would need very careful tracking of the constants in order to lead to robust estimates. Our results are much weaker in many respects but they provide near-diagonal bounds which are robust in the sense that the constants in the bounds stay uniform as $\alpha \to 2-$. In this way, the corresponding near-diagonal bounds for second order operators can be recovered by taking the limit as $\alpha \to 2-$. Since the proofs do not rely on any estimate of the heat kernel $p^{(\alpha)}_\Omega$ we have a second advantage. We can establish rotational bounds for the Green function in situations where the heat kernel does not satisfy rotational bounds, see \Cref{ex:examples} below.

Let us explain our set-up and the main results. Throughout the paper, we always assume $k$ to be a measurable symmetric function on $\R^d \times \R^d \setminus \operatorname{diag} \to [0, \infty)$ satisfying 
\begin{align} \label{eq:Levy-cond}
\sup_{x \in \Rd} \int_{\Rd} \left( 1 \land |y-x|^2 \right) k(x,y) \dy  < +\infty.
\end{align}

\begin{example} \label{ex:alpha-stable}
	An important example satisfying \eqref{eq:Levy-cond} is given by
	\begin{align*}
	k^{(\alpha)}(x, y) = \tfrac{2^\alpha \Gamma(\frac{d+\alpha}{2})}{\pi^{d/2} |\Gamma(\frac{-\alpha}{2})|}  \, |y-x|^{-d-\alpha} \qquad (0 < \alpha < 2).
	\end{align*}
Note that the ratio $\tfrac{2^\alpha \Gamma(\frac{d+\alpha}{2})}{\pi^{d/2} |\Gamma(\frac{-\alpha}{2})|} / (2-\alpha)$ remains bounded as $\alpha \rightarrow 2$. Since we do not care about the behavior of this constant for $\alpha \to 0$, we will often just use $2-\alpha$ as a multiplicative constant.
\end{example}

Let $\Omega \subset \Rd$ be open. For measurable functions $u, v: \Omega \to \R$ we define the expression
\begin{align} \label{eq:bilinear_form_Omega}
\mathcal{E}_\Omega (u, v) 
= \int_\Omega \int_\Omega (u(y) - u(x))(v(y) - v(x)) \, k(x,y) \dy \dx,
\end{align}
provided that is it finite. If $k=k_{\alpha}$ from above, then we write $\mathcal{E}_{\Omega}^{\alpha}$ instead of $\mathcal{E}_{\Omega}$. Moreover, we set $\mathcal{E}(u, v) = \mathcal{E}_{\Rd}(u, v)$ resp. $\mathcal{E}^{\alpha}(u, v) = \mathcal{E}_{\Rd}^{\alpha}(u, v)$. Note that \eqref{eq:Levy-cond} implies that the double-integral \eqref{eq:bilinear_form_Omega} converges absolutely for given $u, v \in C_c^\infty(\Omega)$ (see \cite[Example 1.2.1]{FOT11}). We define an operator
\begin{equation*}
L\varphi(x) = 2 \,p.v. \int_{\Rd} (\varphi(y)-\varphi(x)) k(x, y) \dy.
\end{equation*}
We refer the reader to \Cref{sec:prelims} for definition of function spaces associated to $k$.

\begin{definition}\label{def:Green-function}
	Let $\Omega$ be a bounded open set in $\Rd$. A measurable function $G : \R^d \times \Rd \rightarrow [0, +\infty]$ is called a {\it Green function} of $L$ on $\Omega$ if $G(\cdot, y) \in L^{1}(\Omega)$ for each $y \in \Omega$, $G=0$ $\mathrm{a.e.}$ on $(\Rd \times \Rd) \setminus (\Omega \times \Omega)$, and
	\begin{align} \label{eq:Green_function}
	\int_{\Rd} G(x, y) \psi(x) \dx = \varphi(y)
	\end{align}
for every $y \in \Omega$, $\varphi \in H_{\Omega}(\Rd; k) \cap C(\Omega)$, and $\psi \in L^\infty(\Omega)$ with
\begin{equation} \label{eq:DP}
\mathcal{E}(\varphi, v) = \langle \psi, v \rangle \quad\text{for all } v \in H_{\Omega}(\Rd; k) \,.
\end{equation}
\end{definition}

\begin{remark*}{ \ } \vspace*{-2ex}
\begin{enumerate}
	\item \Cref{def:Green-function} is analogous to the one in \cite[Section 6]{LSW63}. 
	\item The above definition corresponds to the usual definition $G(x,y)$ $= \int_0^{\infty} p_\Omega(t;x,y) \dt$, where $ p_\Omega $ is the corresponding heat kernel. This can be easily seen in case of a L\'evy process $(X_t)_{t \geq 0}$ in $\Rd$ with an infinitesimal generator $L$, where the Green function $G$ satisfies
	\begin{align*}
	\langle G(\cdot, y), \psi \rangle
	&= \langle -LG(\cdot, y), \varphi \rangle
	= \left\langle -L\int_0^\infty p_t(\cdot, y)\dt, \varphi \right\rangle \\
	&= \left\langle -\int_0^\infty Lp_t(\cdot, y)\dt, \varphi \right \rangle
	= \left\langle -\int_0^\infty \partial_t p_t(\cdot, y)\dt, \varphi \right\rangle
	= \varphi(y)
	\end{align*}
for every $y \in \Omega$, $\varphi \in H_{\Omega}(\Rd; k) \cap C(\Omega)$, and $\psi \in L^{\infty}(\Omega)$ with \eqref{eq:DP}. In other words, a Green function solves the equation $-LG(\cdot, y) = \delta_y$, where $\delta_{y}$ is the Dirac measure at the point $y$. In case of $k^{(\alpha)}$ from \Cref{ex:alpha-stable}, the function $G$ is nothing but the Green function of the fractional Laplacian $(-\Delta)^{\alpha/2}$ in the domain $\Omega$.
\end{enumerate}
\end{remark*}

Our first result concerns the existence, uniqueness, and some uniform estimates of the Green function, which we establish under very mild assumptions on the kernel $k$. To this end, we formulate the following conditions $k$. Let $\cl, \cU > 0$ be given.

\begin{enumerate}
	\item [\mylabel{assum:E-lower}{$\mathcal{E}_\geq$}] 
	For every ball $B = B_r(x_0) \subset \Rd$ and every function $u \in L^2(B)$
	\begin{align*}
	\mathcal{E}_B (u, u) \geq \cl \mathcal{E}^{\alpha}_B (u, u) \,.
	\end{align*}

	\item [\mylabel{assum:U1}{U1}]
	For all $r > 0$
	\begin{align*}
	\sup_{x \in \Rd} \int_{\Rd} \left( r^2 \land |y-x|^2 \right) k(x,y) \dy
	\leq \cU r^{2-\alpha} \,.
	\end{align*}
\end{enumerate}

Note that \eqref{assum:E-lower} implies $\mathcal{E}(u, u) \geq \cl \mathcal{E}^{\alpha}(u, u)$ for every function $u \in L^{2}(\Rd)$. Indeed, it follows from $\mathcal{E}(u, u) \geq \mathcal{E}_{B}(u, u) \geq \cl \mathcal{E}^{\alpha}_{B}(u, u)$ and the monotone convergence theorem.

\begin{theorem} [Existence and uniqueness] \label{thm:existence}
	Let $0<\alpha_0 \leq \alpha<2$, $\cl, \cU > 0$, and assume that $k$ satisfies \eqref{assum:E-lower}. Let $\Omega$ be open and bounded. Then there exists a Green function $G$ of $L$ on $\Omega$. 
	Moreover, $G$ satisfies the following properties: for every $y \in \Omega$,
	\begin{align}
	&G(\cdot, y) \in W^{\beta/2, q}_\Omega(\Rd) 
	\quad\text{for all}~ \beta \in (0, \alpha) ~\text{and}~ q \in [1, d/(d-\alpha/2)), \nonumber \\
	&G(\cdot, y) \in L^{d/(d-\alpha)}_\mathrm{weak} (\Rd) \quad\text{with}~ 
	[G(\cdot, y)]_{L^{d/(d-\alpha)}_\mathrm{weak} (\Rd)} \leq C, \label{eq:weak_Lp}
	\end{align}
	where $C$ depends only on $d$, $\cl$, and $\alpha_0$, but not on $\alpha$, $y$, $\Omega$. Furthermore, if in addition $k$ satisfies \eqref{assum:U1}, then the Green function of $L$ on $\Omega$ is unique.
\end{theorem}

We refer the reader to \Cref{sec:prelims} for definition of the function spaces $W^{\beta/2, q}_\Omega(\Rd)$ and $L^{d/(d-\alpha)}_\mathrm{weak} (\Rd)$.

The next results provide pointwise upper and lower bounds for Green functions which are robust in the sense that the constants stay uniform as $\alpha \to 2-$. In order to state the results, we formulate more assumptions on the kernel $k$. Let $\cK, \cH> 0$ be given positive constants.
\begin{enumerate}
	\item [\mylabel{assum:U2}{U2}]
	For almost every $x, y \in \Rd$,
	\begin{equation*}
	k(x, y) \leq \cK (2-\alpha) |y-x|^{-d-\alpha} \,.
	\end{equation*}
	
	\item [\mylabel{assum:H}{H}] 
	(Annulus Harnack inequality for Green function) Let $G$ be a Green function of $L$ on $\Omega$. For every ball $B_{2r}(y) \Subset \Omega$ 
	\begin{align*}
	\sup_{B_{2r}(y) \setminus B_r(y)} G(\cdot, y) 
	\leq \cH \inf_{B_{2r}(y) \setminus B_r(y)} G(\cdot, y).
	\end{align*}
	
\end{enumerate}

We will discuss these conditions in detail below. Let us recall that we assume $d \geq 3$, which in particular implies $d > \alpha$. We do not discuss possible extensions, e.g., the cases $d=2, \alpha \in (0,2)$ or $d=1, \alpha \in [1, 2)$.

\begin{theorem} [Upper bound] \label{thm:upper-bound}
Let $0 < \alpha_0 \leq \alpha < 2$, $\cl, \cK > 0$, and assume that $k$ satisfies \eqref{assum:E-lower} and \eqref{assum:U2}. Then the Green function $G$ of $L$ on $\Omega$ satisfies
\begin{align} \label{eq:upper-bound}
G(x,y) \leq C |x-y|^{\alpha-d} \quad\text{for all} ~ x, y \in \Omega
\end{align}
for some constant $C > 0$ depending only on $d$, $\cl$, $\cK$, and $\alpha_0$, but not on $\alpha$, $\Omega$.
\end{theorem}

\begin{theorem} [Lower bound] \label{thm:lower-bound}
Let $0 < \alpha_0 \leq \alpha < 2$, $\cl, \cU, \cH > 0$, and assume that $k$ satisfies \eqref{assum:E-lower}, \eqref{assum:U1}, \eqref{assum:H}. Then the Green function $G$ of $L$ on $\Omega$ satisfies
\begin{align} \label{eq:lower-bound}
G(x,y) \geq C |x-y|^{\alpha-d} \quad\text{for all} ~ x, y \in \Omega 
~\text{with}~ |x-y| \leq \mathrm{dist}(y, \partial \Omega)/2,
\end{align}
for some constant $C > 0$ depending only on $d$, $\cl$, $\cU$, $\cH$, and $\alpha_0$, but not on $\alpha$, $\Omega$.
\end{theorem}

We remark that the assumption \eqref{assum:U2} in \Cref{thm:upper-bound} cannot be replaced by \eqref{assum:U1}, at least when $\alpha \leq (d-1)/2$. This can be seen from the observation that the singular measure
\begin{align*}
\mu(x, \d y) = (2-\alpha) \sum_{i=1}^d |y_i-x_i|^{-1-\alpha} \prod_{j \neq i} \delta_{x_i}(\d y_i)
\end{align*}
satisfies \eqref{assum:E-lower} and \eqref{assum:U1}, but the corresponding Green function does not satisfy \eqref{eq:upper-bound}. See the last paragraph of \cite{BoSz05}.

Let us discuss the aforementioned conditions on $k$. 
\begin{enumerate}[(i)]
\item
Let $k^{(\alpha)}(x, y)$ be the function defined in \Cref{ex:alpha-stable}. Let $\alpha_0 \in (0,2)$ and assume $\alpha \in [\alpha_0, 2)$. Then there are constants $\cl$, $\cU$, $\cK$, and $\cH$ such that all of the aforementioned conditions hold for every $\alpha \in [\alpha_0, 2)$. This means that the conditions hold in a robust fashion as $\alpha \to 2-$, i.e., constants do not degenerate.

\item
If $k$ satisfies the condition \eqref{assum:U2} for $\alpha \in [\alpha_0, 2)$, then the condition \eqref{assum:U1} holds for $\cU = 2|\mathbb{S}^{d-1}| \alpha_0^{-1}\cK$. 
\end{enumerate}

It may not be always easy to check whether the condition \eqref{assum:H} holds for a given function $k$. However, it is known that the condition \eqref{assum:H} holds for a large class of functions $k$. Let us consider the following condition. Let $\cJ > 0$ be given.
\begin{enumerate}
\item [\mylabel{assum:UJS}{UJS}]
For almost every $x, y \in \Rd$ with $x\neq y$ and $0 < r \leq |x-y|/2$
\begin{equation*}
k(x,y) \leq \cJ \fint_{B_r(x)} k(z, y) \,\d z.
\end{equation*}
\end{enumerate}

\begin{enumerate}[(i)]
\item[(iii)]
It is known \cite[Theorem 11.14]{Sch20} that the Harnack inequality holds under the conditions \eqref{assum:E-lower}, \eqref{assum:U2}, and \eqref{assum:UJS}. The condition \eqref{assum:H} follows from this result applied to a Green function together with a standard covering argument.
\end{enumerate}

\begin{corollary} \label{cor:bounds}
Let $0 < \alpha_0 \leq \alpha < 2$, $\cl, \cK, \cJ > 0$, and assume that $k$ satisfies \eqref{assum:E-lower}, \eqref{assum:U2}, and \eqref{assum:UJS}. Then the Green function of $L$ on $\Omega$ satisfies \eqref{eq:upper-bound} and \eqref{eq:lower-bound}.
\end{corollary}

The last result we provide is  about the symmetry of the Green function.

\begin{theorem} [Symmetry] \label{thm:symm}
Let $\alpha \in (0,2)$ and $\cl, \cU > 0$. Assume that $k$ satisfies \eqref{assum:E-lower} and \eqref{assum:U1}. Then the Green function $G$ of $L$ on $\Omega$ is symmetric, i.e., $G(x, y) = G(y, x)$ for all $x, y \in \Omega$.
\end{theorem}

Let us provide some examples of $k$.

\begin{example} \label{ex:examples}
\begin{enumerate}[(i)]
\item
Let $\mathcal{C} = \lbrace h \in \Rd : |\frac{h}{|h|} \cdot e_d| > c \rbrace$, $c \in (0,1)$, be a double cone and define a function $k(x, y) = (2-\alpha) |y-x|^{-d-\alpha} \chi_{\mathcal{C}}(y-x)$. Let $p(t; x, y)$ be the heat kernel associated with $\mathcal{E}$. As explained in \cite[Example 1.2]{CKW20}, the classical heat kernel estimate fails to hold. In fact, $p(t; x,y)$ has no bounds which are rotationally symmetric in space. Indeed, by \cite[Equation (2.4)]{BKK15}, we know that
\begin{equation*}
\int_{\Rd} f(y) k(x, y) \dy
= \lim_{t \searrow 0} \frac{1}{t} \int_{\Rd} f(y) p(t;x, y) \dy
\end{equation*}
for all $f \in C_c(\Rd)$ and $x \in \Rd \setminus \mathrm{supp} \, f$. Thus, we obtain
\begin{equation*}
k(x, y) = k(y-x) = \lim_{t \searrow 0} \frac{p(t; x, y)}{t}.
\end{equation*}
This implies that
\begin{equation*}
\lim_{t \searrow 0} \frac{p(t; 0, e_1)}{p(t; 0, e_d)} = 0,
\end{equation*}
which shows that $p$ cannot have rotationally symmetric bounds. However, a Green function $G$ of $L$ on $\Omega$ has rotationally symmetric bounds \eqref{eq:upper-bound} and \eqref{eq:lower-bound} since $k$ satisfies \eqref{assum:E-lower}, \eqref{assum:U2}, and \eqref{assum:UJS}. Indeed, the conditions \eqref{assum:U2} and \eqref{assum:UJS} are satisfied obviously. See \cite[Theorem 1.11]{DyKa20} for the condition \eqref{assum:E-lower}.

\item
Let us provide another example of $k$ whose Green function satisfies isotropic bounds but the heat kernel does not. Let $a > 1$, then there exists $c > 0$ such that every annulus $B_{a^{-n+1}} \setminus B_{a^{-n}}$, $n=0, 1, \dots$, contains a ball $B(x_n, ca^{-n})$. Let $\tilde{k}(z)$ be a function such that
\begin{equation*}
\tilde{k}(z) = 
\begin{cases}
(2-\alpha) |z|^{-d-\alpha} &\text{if}~ z \in B(x_n, ca^{-n}) \cup B(-x_n, ca^{-n}), \\
0 &\text{otherwise},
\end{cases}
\end{equation*}
and let $k$ be a function satisfying $k(x, y) = k(x-y) \asymp \tilde{k}(x-y)$. Then, the heat kernel does not have isotropic bounds in the same spirit as in the previous example. However,  it is known \cite[Corollary 12.5]{Sch20} that $k$ satisfies \eqref{assum:E-lower}, \eqref{assum:U2}, and \eqref{assum:UJS}.

\item
Consider a non-degenerate translation-invariant kernel $k$ which is symmetric in the sense that $k(z) = k(-z)$ and homogeneous of degree $-d-\alpha$. Note that the non-degeneracy gives \eqref{assum:E-lower}. If $k$ satisfies \eqref{assum:U2}, then the relative Kato condition holds if and only if the Harnack inequality holds (see \cite{BoSz05} for the relative Kato condition). Moreover, in this setting the condition \eqref{assum:UJS} is stronger than the relative Kato condition (see \cite[Theorem B.6]{Sch20}).

\item
Let us consider a function $k$ satisfying
\begin{equation*}
\Lambda^{-1}(2-\alpha) \left( \chi_{V^\Gamma[x]}(y) + \chi_{V^\Gamma[y]}(x) \right) |y-x|^{-d-\alpha} \leq k(x, y) \leq \Lambda (2-\alpha) |y-x|^{-d-\alpha}, 
\end{equation*}
where $V^\Gamma[x] = x + \Gamma(x)$ is a double cone, with apex at $x$ and a fixed opening, that might rotate arbitrarily from point to point (see \cite{BKS19} for the precise definition). Note that $k$ is non-translation-invariant in general and \Cref{ex:examples} (i) is a special case of this example which is translation-invariant.

It is proved \cite{BKS19} that \eqref{assum:E-lower} holds for all admissible configurations. Moreover, there are some configurations that fulfill the condition \eqref{assum:UJS}, see \cite[Lemma 12.2]{Sch20}. However, there also exists an example that does not satisfy the condition \eqref{assum:UJS}. See \cite[Section 12.1]{Sch20}.

\item
Another example of a non-translation-invariant measure is given by
\begin{equation*}
k(x,y) \dy = (2-\alpha) \frac{a(x, y)}{|y-x|^{d+\alpha}} \dy
\end{equation*}
with a measurable function $a$, which is uniformly bounded above and below away from 0. It is proved \cite{BaLe02} that the Harnack inequality holds for nonnegative harmonic functions, from which \eqref{assum:H} follows.

\item
It is proved \cite{ChSi20} that the condition \eqref{assum:E-lower} is implied by the following mild condition: there are constants $\delta \in (0,1)$ and $\lambda > 0$ such that for every ball $B \subset \Rd$ and every point $x \in B$,
\begin{equation*}
|\lbrace y \in B : k(x,y) \geq \lambda (2-\alpha) |y-x|^{-d-\alpha} \rbrace | \geq \delta |B|.
\end{equation*}
\end{enumerate}
\end{example}

\begin{remark}
Let $b \in (0,1)$ and
\begin{align*}
\Gamma = \lbrace (x_1, x_2) \in \R^2 : |x_2| \geq |x_1|^b
~\text{or}~ |x_1| \geq |x_2|^b \rbrace.
\end{align*}
We consider a function $k(z) = (2-\alpha) {\bf 1}_{\Gamma \cap B_1} |z|^{-2-\gamma}$,
where $\gamma = \alpha - 1 + 1/b$. 
It is proved in \cite[Example 6.15]{DyKa20} that $k$ satisfies \eqref{assum:E-lower} and \eqref{assum:U1} with $\alpha$ unchanged. Note that \eqref{assum:U2} is satisfied with $\gamma$, but not with $\alpha$. Therefore, for this example we obtain the existence result only.
\end{remark}

Let us review some related results. With the help of techniques from classical potential theory \cite{LSW63,GrWi82,Zha86} establish sharp Green function estimates for second order operators on bounded $C^{1,1}$ domains. \cite{Kul97,ChSo98} extend these results to operators of fractional order by studying symmetric $\alpha$-stable processes ($0 < \alpha < 2$) on bounded $C^{1,1}$ domains. Their proofs rely on  explicit formulas for the Green functions on balls. \cite{CKS10a} uses fine estimates of the two-sided Dirichlet heat kernel in order to prove a similar result. Dirichlet heat kernel estimates and Green function estimates have been extended to a wide class of operators of fractional order and related Markov processes. See \cite{CKSV12, CKS14, BGR14, KiMi14, GKK20} for a selection of recent results in this direction. As mentioned in the beginning of the introduction, these estimates are sharp for fixed $\alpha$ but the constants in the estimates may degenerate as $\alpha \to 2-$. 
In the case of symmetric $\alpha$-stable processes, i.e., for translation invariant operators with appropriate symmetric measures, \cite{Che99,ChSo04} do establish Green function estimates that are robust as $\alpha \to 2-$. Both works use explicit formulas for the Green function on balls, thus the restricted scope. However, the results of \cite{Che99,ChSo04} include global estimates up to the boundary of the considered domain. It is an interesting task to prove such global bounds in the framework of our work.

Using a different approach based on variational techniques as in \cite{LSW63,GrWi82}, the Green function for nonlocal operators of fractional order with rough coefficients is studied in \cite{KaSt02,KMS15,CaSi18}. The authors of \cite{CaSi18} obtain the near-diagonal estimates for Green functions and the boundary Harnack principle in the framework of nonlinear nonlocal operators modeled after the fractional $p$-Laplacian. While these results are obtained for kernels comparable to rotationally symmetric ones, our results allow for kernels that may not have rotationally symmetric bounds. 

\begin{remark}
As a byproduct of this work, we correct a serious mistake in \cite{KaSt02}. The proof of Proposition 3.3 is incorrect. One cannot apply Theorem 2.1 because the function $G_\rho$ does not vanish on the complement of $\Omega \setminus B_\rho$. 
\end{remark}

The paper is organized as follows. 
In \Cref{sec:prelims} we introduce function spaces and recall some embedding theorems. Moreover, we provide algebraic inequalities for later use. In \Cref{sec:existence} we prove our first main theorem, \Cref{thm:existence}, which establishes the existence and uniqueness results with the uniform estimates on $L_{\mathrm{weak}}^{d/(d-\alpha)}$. This estimate plays a crucial role in \Cref{sec:upper-bounds}, where the local boundedness result and the proof of \Cref{thm:upper-bound} are provided. \Cref{sec:lower-bounds} is devoted to the proof of the pointwise lower bounds of Green functions, i.e., \Cref{thm:lower-bound}. Finally, the symmetry of the Green function is proved in \Cref{sec:symmetry}.

{\bf Acknowledgement}: The authors thank the anonymous referee and Marvin Weidner for helpful comments that led to an improvement of the paper.

\section{Preliminaries} \label{sec:prelims}

In this section, we recall function spaces and embedding theorems that will be used in the sequel. We will also provide some algebraic inequalities in the end of this section for later use.

\begin{definition} [Function spaces]
Let $\Omega \subset \Rd$ be an open set in $\Rd$ and $k$ be a measurable function satisfying \eqref{eq:Levy-cond}. We define the following linear spaces:
\begin{enumerate} [(i)]
\item
$H(\Rd; k) = \lbrace u \in L^2(\Rd) : \mathcal{E}(u, u) < +\infty \rbrace$.
\item
$H_\Omega(\Rd; k) 
= \lbrace u \in H(\Rd; k): u = 0 ~\mathrm{a.e.} ~\text{on}~ \Rd \setminus \Omega \rbrace$.
\end{enumerate}
If $k=k^{(\alpha)}$ as in \Cref{ex:alpha-stable}, then we write $H_\Omega(\Rd; k^{(\alpha)}) = H^{\alpha/2}_{\Omega}(\Rd)$. In a similar manner, we introduce the space $W^{\alpha/2, p}_\Omega(\Rd)$ for $p \geq 1$:
\begin{enumerate} [(i)]
\item[(iii)]
$W^{\alpha/2, p}_\Omega(\Rd) = \lbrace u \in W^{\alpha/2, p}(\Rd): u = 0 ~\mathrm{a.e.} ~\text{on}~ \Rd \setminus \Omega \rbrace$.
\end{enumerate}
We will also make use of the weak $L^p$ spaces:
\begin{enumerate}[(i)]
\item[(iv)]
$L^p_\mathrm{weak}(\Omega) 
= \lbrace u : [u]_{L^p_\mathrm{weak}(\Omega)} < +\infty \rbrace$, where
\begin{align*}
[u]_{L^p_\mathrm{weak}(\Omega)} 
= \sup_{t > 0} t |\lbrace x \in \Omega : |u(x)| > t \rbrace|^{1/p}.
\end{align*}
\end{enumerate}
\end{definition}

It is known \cite{FKV15} that the spaces $H(\Rd; k)$ and $H_\Omega(\Rd; k)$, endowed with the norm $\|u\|_{H(\Rd; k)} = (\|u\|_{L^2(\Rd)}^2 + \mathcal{E}(u,u))^{1/2}$, are Hilbert spaces even when $k$ is not symmetric. Similarly, $W^{\alpha/2, p}_{\Omega}(\Rd)$ is a Banach space with the fractional Sobolev norm
\begin{equation*}
\begin{split}
\|u\|_{W^{\alpha/2,p}(\Rd)} 
&= (\|u\|_{L^p(\Rd)}^p + [u]_{W^{\alpha/2, p}(\Rd)}^p)^{1/p} \\
&= \left( \int_{\Rd} |u(x)|^p \dx + (2-\alpha)\int_{\Rd} \int_{\Rd} \frac{|u(y)-u(x)|^p}{|x-y|^{d+\frac{\alpha}{2}p}} \dy \dx \right)^{1/p}.
\end{split}
\end{equation*}

Note that, for simplicity, we choose $(2-\alpha)$ as a constant in the seminorm because, in our results, we only trace the behavior of constants for $\alpha \to 2-$. Let us collect some embedding theorems for the spaces introduced above. 

\begin{proposition} \label{prop:cont-embedding}
Let $\Omega \subset \Rd$ be open and bounded. If $0<\beta<\alpha<2$ and $p>q\geq1$, then $W^{\alpha/2,p}_{\Omega}(\Rd) \subset W^{\beta/2,q}_{\Omega}(\Rd)$.
\end{proposition}

\begin{proof}
Let $u \in W^{\alpha/2,p}_{\Omega}(\Rd)$. By H\"older's inequality, we have
\begin{equation} \label{eq:embedding-Lq}
\|u\|_{L^{q}(\Rd)} = \|u\|_{L^{q}(\Omega)} \leq |\Omega|^{\frac{1}{q}-\frac{1}{p}} \|u\|_{L^{p}(\Omega)} = |\Omega|^{\frac{1}{q}-\frac{1}{p}} \|u\|_{L^{p}(\Rd)} < \infty.
\end{equation}
To estimate $[u]_{W^{\beta/2,q}(\Rd)}$, let us fix a point $x_{0} \in \Omega$ and consider a large Ball $B_{R}(x_{0})$ such that $\Omega \subset B_{R}(x_{0})$ and $\mathrm{dist}(\Omega, \partial B_{R}(x_{0})) \geq 2\mathrm{diam}(\Omega)$. Then, by \cite[Lemma 4.6]{Coz17} we have
\begin{equation} \label{eq:embedding-I1}
I_{1} := \int_{\Omega} \int_{B_{R}(x_{0})} \frac{|u(x)-u(y)|^{q}}{|x-y|^{d+\frac{\beta}{2}q}} \dy \dx \leq C \left( \int_{\Omega} \int_{B_{R}(x_{0})} \frac{|u(x)-u(y)|^{p}}{|x-y|^{d+\frac{\alpha}{2}p}} \dy \dx \right)^{q/p} \leq C [u]_{W^{\alpha/2, p}(\Rd)}^{q}
\end{equation}
for some $C = C(d, \alpha, \beta, p, q, R, |\Omega|)>0$. Recalling $u=0$ outside $B_{R}(x_{0})$, we have
\begin{equation} \label{eq:embedding-I2}
I_{2} := \int_{\Omega} \int_{\Rd \setminus B_{R}(x_{0})} \frac{|u(x)-u(y)|^{q}}{|x-y|^{d+\frac{\beta}{2}q}} \dy \dx = \int_{\Omega} |u(x)|^{q} \int_{\Rd \setminus B_{R}(x_{0})} \frac{\mathrm{d}y}{|x-y|^{d+\frac{\beta}{2}q}} \dx.
\end{equation}
If $x \in \Omega$ and $y \in \Rd \setminus B_{R}(x_{0})$, then $|x-y| \geq |y-x_{0}| - |x-x_{0}| \geq \frac{1}{2} |y-x_{0}|$. Thus, it follows from \eqref{eq:embedding-Lq} and \eqref{eq:embedding-I2} that
\begin{equation} \label{eq:embedding-I22}
I_{2} \leq C \int_{\Omega} |u(x)|^{q} \int_{\Rd \setminus B_{R}(x_{0})} \frac{\mathrm{d}y}{|y-x_{0}|^{d+\frac{\beta}{2}q}} \dx \leq C \|u\|_{L^{q}(\Omega)}^{q} \leq C \|u\|_{L^{p}(\Rd)}^{q},
\end{equation}
where $C$ is a constant depending on $d$, $\beta$, $p$, $q$, $R$, and $|\Omega|$. Combining \eqref{eq:embedding-I1} and \eqref{eq:embedding-I22} yields
\begin{equation} \label{eq:embedding-W}
[u]_{W^{\beta/2,q}(\Rd)}^{q} \leq 2 \int_{\Omega} \int_{\Rd} \frac{|u(x)-u(y)|^{q}}{|x-y|^{d+\frac{\beta}{2}q}} \dy \dx \leq 2(I_{1} + I_{2}) \leq C \|u\|_{W^{\alpha/2, p}(\Rd)}^{q} < \infty.
\end{equation}
Therefore, the desired result follows from \eqref{eq:embedding-Lq} and \eqref{eq:embedding-W}.
\end{proof}

We recall the fractional Sobolev inequalities, see \cite{BBM02, MaSh02, CoTa04}.

\begin{theorem} [Sobolev inequality] \label{thm:Sobolev}
Let $0<\alpha_0 \leq \alpha<2$ and assume $d > \alpha$.
\begin{enumerate} [(i)]
\item
There exists a constant $C$,
depending only on $d$ and $\alpha_0$, such that for any measurable and compactly 
supported function $u : \Rd \rightarrow \R$ we have
\begin{align*}
\|u\|_{L^{2^\ast} (\Rd)}^2 \leq C(2-\alpha) \int_{\Rd} \int_{\Rd} 
\frac{|u(y) - u(x)|^2}{|y-x|^{d+\alpha}} \dy \dx,
\end{align*}
where $2^\ast = 2d / (d-\alpha)$.

\item
There exists a constant $C$,
depending only on $d$ and $\alpha_0$, such that for any $u \in H^{\alpha/2}(B_r)$
\begin{align*}
\|u\|_{L^{2^\ast} (B_r)}^2 \leq C(2-\alpha) \int_{B_r} \int_{B_r} \frac{|u(y) - u(x)|^2}{|y-x|^{d+\alpha}} \dy \dx + Cr^{-\alpha} \| u\|_{L^2(B_r)}^2.
\end{align*}
\end{enumerate}
\end{theorem}

We next recall the embedding theorems for the weak $L^p$ spaces. See, for example, \cite[Theorem 2.18.8]{KJF77} for the proof.

\begin{theorem}
Let $\Omega$ be a bounded open set in $\Rd$ and let $p \geq 1$. Then, $L^p(\Omega) \subset L^p_\mathrm{weak}(\Omega)$ and
\begin{align*}
[u]_{L^p_\mathrm{weak}(\Omega)} \leq \|u\|_{L^p(\Omega)}.
\end{align*}
If $1 \leq q <p$, then $L^p_\mathrm{weak}(\Omega) \subset L^q (\Omega)$ and
\begin{align} \label{eq:weak_Lq_embedding}
\|u\|_{L^q(\Omega)} \leq \left( p/(p-q) \right)^{1/q} |\Omega|^{1/q-1/p} [u]_{L^p_\mathrm{weak}(\Omega)}.
\end{align}
\end{theorem}

We end this section with some algebraic inequalities that will be useful throughout the paper.

\begin{lemma} [Lemma 3.2 in \cite{KaSt02}] \label{lem:inequality1}
For $a, b \geq 0$ and $s \in (0,1)$
\begin{align*}
\frac{8}{2^{1/s}} (1-s)^{-2} ((1+b)^{(1-s)/2} - (1+a)^{(1-s)/2})^2 \leq (b-a) (b(1+b^s)^{-1/s} - a(1+a^s)^{-1/s}).
\end{align*}
\end{lemma}

\begin{lemma} \label{lem:inequality1-lower}
For $a, b \geq 0$ and $s \in (0, 1)$
\begin{equation*}
|b-a| \leq \frac{2}{1-s} \left| (1+b)^{\frac{1-s}{2}} - (1+a)^{\frac{1-s}{2}} \right| \max\left\lbrace (1+a)^{\frac{1+s}{2}}, (1+b)^{\frac{1+s}{2}} \right\rbrace.
\end{equation*}
\end{lemma}

\begin{proof}
Assume that $b > a$, then we have
\begin{equation*}
\frac{(1+b)^{\frac{1-s}{2}}-(1+a)^{\frac{1-s}{2}}}{b-a} = \frac{1}{b-a} \int_a^b \frac{1-s}{2} (1+t)^{-\frac{1+s}{2}} \dt \geq \frac{1-s}{2} (1+b)^{-\frac{1+s}{2}}.
\end{equation*}
The case $a > b$ can be proved in the same way.
\end{proof}

\begin{lemma} \label{lem:inequality2}
For $a, b > 0$ 
\begin{align*}
(b-a)(a^{-1} - b^{-1}) \geq (\log b - \log a)^2.
\end{align*}
\end{lemma}

One can easily prove \Cref{lem:inequality2} by setting $a = e^x$, $b = e^y$, and then applying the Taylor expansion.

\begin{lemma} [Lemma 3.7 in \cite{Kas07}] \label{lem:inequality3}
For $a, b \geq 0, \eta_1, \eta_2 \geq 0$, and $q > 1$
\begin{align*}
(b-a)(\eta_2^2 b^{q-1} - \eta_1^2 a^{q-1})
\geq \frac{q-1}{32q^2} \left( \eta_2 b^\frac{q}{2} - \eta_1 a^\frac{q}{2} \right)^2
- 2(1 \lor (q-1)^{-1}) (\eta_2 - \eta_1)^2 (b^q + a^q).
\end{align*}
\end{lemma}

\begin{lemma} \label{lem:inequality4}
For $a,b,x,y \in \R$
\begin{align*}
(b-a)(by^2 - ax^2) \geq \frac{1}{4} (b-a)^2 (y^2 + x^2) - 4(b^2 + a^2) (y-x)^2.
\end{align*}
\end{lemma}

\Cref{lem:inequality4} follows from the equality $by^2 - ax^2 = \frac{1}{2} (b-a)(y^2+x^2) + \frac{1}{2} (b+a)(y^2 -x^2)$ and Young's inequality.

\section{Existence and uniqueness of the Green function} \label{sec:existence}

In this section we establish the existence and uniqueness results, \Cref{thm:existence}, by adapting the ideas in \cite{GrWi82,BeFr02,KaSt02}. We start by constructing regularized Green functions.

\begin{lemma}
Let $\alpha \in (0, 2)$, $\cl > 0$, and assume that $k$ satisfies \eqref{assum:E-lower}. Let $\Omega$ be a bounded open set. 
For every $y_0 \in \Omega$ and $\rho > 0$ with $B_\rho(y_0) \subset \Omega$,
there exists a unique nonnegative function $G_\rho(\cdot, y_0) \in H_\Omega(\Rd;k)$ satisfying
\begin{align} \label{eq:regularized-G}
\mathcal{E}(G_\rho(\cdot, y_0), \varphi) = \fint_{B_\rho(y_0)} \varphi(x) \dx \quad\text{for all}~ \varphi \in H_\Omega(\Rd;k).
\end{align}
\end{lemma}

\begin{proof}
It is clear that $\mathcal{E}$ is a continuous bilinear form on $H_\Omega(\Rd;k)$
and that the map $\varphi \mapsto \fint_{B_\rho(y_0)} \varphi$ is a continuous linear functional on $H_\Omega(\Rd;k)$. 
Moreover, coercivity of $\mathcal{E}$ follows from 
\cite[Lemma 2.9]{FKV15} by means of the assumption \eqref{assum:E-lower}. By the Lax--Milgram Theorem there exists a unique function $G_\rho(\cdot, y_0) \in H_\Omega(\Rd;k)$ satisfying \eqref{eq:regularized-G}. 
The nonnegativity of $G_\rho$ follows by a standard argument, see, for instance, the proof of \cite[Theorem 1.1]{GrWi82}.
\end{proof}

To pass the limit from \eqref{eq:regularized-G}, we provide uniform estimates of $G_\rho(\cdot, y_0)$ in $W^{\beta/2, q}_\Omega(\Rd)$
for all $\beta \in (0, \alpha)$ and $q \in [1, d/(d-\alpha/2))$, which are independent of $\rho$ and $y_0$. Note that we do not care about $\alpha$ dependence of the constant $C$ in the following proposition.

\begin{proposition} \label{prop:uniform_estimates}
Let $\alpha \in (0, 2)$, $\cl > 0$, and assume that $k$ satisfies \eqref{assum:E-lower}. Let $\Omega$ be a bounded open set. Then, for any $\beta \in (0,\alpha)$ and $q \in [1, d/(d-\alpha/2))$,
\begin{equation} \label{eq:seminorm-bq}
[G_\rho(\cdot, y_0)]_{W^{\beta/2, q}(\Rd)} \leq C
\end{equation}
for some constant $C$ independent of $\rho$ and $y_0$.
\end{proposition}

\begin{proof}
Let us write $G_\rho(x) = G_\rho(x,y_0)$. Let
\begin{equation*}
s = \frac{d(1-q)+\frac{\alpha}{2}q}{d-\frac{\alpha}{2}q} \in (0, 1)
\end{equation*}
so that
\begin{equation} \label{eq:s}
\frac{1-s}{2} 2^{\ast} = (1+s) \frac{q}{2-q},
\end{equation}
where $2^{\ast} = \frac{2d}{d-\alpha}$ is the Sobolev exponent. We take $\varphi(x) = G_\rho(x) (1+ G_\rho(x)^s)^{-1/s}$ as a test function in \eqref{eq:regularized-G} and then use \Cref{lem:inequality1} to obtain
\begin{equation*}
c(s) \, \mathcal{E} \left( (1+G_\rho)^\frac{1-s}{2}, (1+G_\rho)^\frac{1-s}{2} \right) \leq \mathcal{E}(G_\rho, \varphi) = \fint_{B_\rho(y_0)} \varphi(x) \dx \leq 1,
\end{equation*}
where the last inequality follows from the fact that $\varphi \in [0,1]$.
Since $k$ satisfies \eqref{assum:E-lower}, we have
\begin{align} \label{eq:seminorm}
\left[ (1+G_\rho)^\frac{1-s}{2} \right]_{W^{\alpha/2, 2}(\Rd)} \leq C
\end{align}
for some $C > 0$. Moreover, since $G_{\rho}$ vanishes outside $\Omega$, the function $(1+G_\rho)^{(1-s)/2}-1$ has a compact support. Therefore, \Cref{thm:Sobolev} (i) implies
\begin{equation} \label{eq:Lp}
\left\| (1+G_\rho)^{\frac{1-s}{2}}-1 \right\|_{L^{2^{\ast}}(\Omega)} = \left\| (1+G_\rho)^{\frac{1-s}{2}}-1 \right\|_{L^{2^{\ast}}(\Rd)} \leq C
\end{equation}
for some positive constant $C$, which may depend on $d, s, \cl, \alpha$.

Let us now estimate
\begin{equation*}
\begin{split}
[G_\rho]_{W^{\beta/2, q}(\Rd)}^q 
&\leq 2(2-\alpha)\int_{\Omega} \int_{B_R(x)} \frac{|G_\rho(y) - G_\rho(x)|^q}{|x-y|^{d+\frac{\beta}{2}q}} \dy \dx \\
&\quad + 2(2-\alpha) \int_{\Omega} \int_{\Rd \setminus B_R(x)} \frac{G_\rho(x)^q}{|x-y|^{d+\frac{\beta}{2}q}} \dy \dx \\
&=: 2I_1 + 2I_2,
\end{split}
\end{equation*}
where $R = \mathrm{diam}(\Omega)$. We first use \Cref{lem:inequality1-lower} and H\"older's inequality with $p=2/q$ to have
\begin{equation} \label{eq:local} 
\begin{split}
I_1
&\leq C \int_{\Omega} \int_{B_R(x)} \frac{|(1+G_\rho(y))^{\frac{1-s}{2}} - (1+G_\rho(x))^{\frac{1-s}{2}}|^q}{|x-y|^{(d+\alpha)\frac{q}{2}}} \frac{(1+G_\rho(x))^{\frac{1+s}{2}q}}{|x-y|^{d(1-\frac{q}{2})-(\alpha-\beta)\frac{q}{2}}} \dy \dx \\
&\leq C \left[ (1+G_\rho)^\frac{1-s}{2} \right]_{W^{\alpha/2, 2}(\Rd)}^q \left( \int_{\Omega} \int_{B_R(x)} \frac{(1+G_\rho(x))^{(1+s)\frac{q}{2-q}}}{|x-y|^{d-(\alpha-\beta)\frac{q}{2-q}}} \dy \dx \right)^{\frac{2-q}{2}}.
\end{split}
\end{equation} 
Using \eqref{eq:s} and \eqref{eq:Lp} yields
\begin{equation} \label{eq:local_Lp}
\int_\Omega \int_{B_R(x)} \frac{(1+G_\rho(x))^{(1+s)\frac{q}{2-q}}}{|x-y|^{d-(\alpha-\beta)\frac{q}{2-q}}} \dy \dx = \frac{2-q}{(\alpha-\beta)q} |\mathbb{S}^{d-1}| R^{(\alpha-\beta)\frac{q}{2-q}} \int_\Omega (1+G_\rho(x))^{\frac{1-s}{2}2^{\ast}} \dx \leq C
\end{equation}
for some constant $C = C(d, \alpha, \alpha-\beta, q, \cl, \mathrm{diam}(\Omega)) > 0$. Thus, it follows from \eqref{eq:seminorm}, \eqref{eq:local}, \eqref{eq:local_Lp}, and the triangle inequality that $I_1 \leq C$ with $C$ independent of $\rho$ and $y_0$.

For $I_2$, we observe
\begin{equation*}
\int_{\Rd \setminus B_R(x)} |x-y|^{-d-\frac{\beta}{2}q} \dy = \frac{2}{\beta q}|\mathbb{S}^{d-1}| R^{-\frac{\beta}{2}q}.
\end{equation*}
Thus, by H\"older's inequality we obtain
\begin{equation*}
I_2 \leq C \int_{\Omega} G_{\rho}(x)^q \dx \leq C \left( \int_{\Omega} G_{\rho}(x)^{\frac{1-s}{2}2^{\ast}} \dx \right)^{\frac{2q}{(1-s)2^{\ast}}}
\end{equation*}
for some constant $C = C(d, \alpha, \beta, q, R, |\Omega|) > 0$. Therefore, it follows from \eqref{eq:Lp} $I_2 \leq C$.

We have estimated $I_1$ and $I_2$ by some constant $C$ independent of $\rho$ and $y_0$, which proves \eqref{eq:seminorm-bq}.
\end{proof}

We are now ready to prove \Cref{thm:existence}.

\begin{proof} [Proof of \Cref{thm:existence}]
Let us first prove that \eqref{eq:weak_Lp} holds for $G_{\rho}=G_{\rho}(\cdot, y_{0})$. We take the same super-level set and the same test function as in the proof of
\cite[Proposition 3.1]{KaSt02}, i.e., for $t > 0$, let $\Omega_{\rho, t} = \lbrace x \in \Rd: 
G_{\rho}(x) > t \rbrace$ and let $\varphi(x) = \max \lbrace 0, 1/t-1/G_{\rho}(x) \rbrace$.
We claim that for all $x, y \in \Rd$
\begin{align} \label{eq:ineq}
(G_{\rho}(y) - G_{\rho}(x)) (\varphi(y) - \varphi(x))
\geq \left( \log \left( \frac{G_{\rho}(y)}{t} \lor 1 \right) 
- \log \left( \frac{G_{\rho}(x)}{t} \lor 1 \right) \right)^2. 
\end{align}
Indeed, when $(x, y) \in \Omega_{\rho, t} \times \Omega_{\rho, t}$, we have $\varphi(x) = 1/t 
- 1/G_{\rho}(x)$ and $\varphi(y) = 1/t - 1/G_{\rho}(y)$. Thus, \Cref{lem:inequality2} gives
\begin{align*}
(G_{\rho}(y) - G_{\rho}(x))(\varphi(y) - \varphi(x)) 
&= (G_{\rho}(y) - G_{\rho}(x)) \left( \frac{1}{G_{\rho}(x)} - \frac{1}{G_{\rho}(y)} \right) \\
&\geq \left( \log G_{\rho}(y) - \log G_{\rho}(x) \right)^2.
\end{align*}
Since $G_{\rho}(x) > t$ and $G_{\rho}(y) > t$ in this case, we arrive at \eqref{eq:ineq}. 
The case $(x, y) \in \Omega_{\rho, t}^c \times \Omega_{\rho, t}^{c}$ is obvious 
because both sides of \eqref{eq:ineq} become 0.
When $(x, y) \in \Omega_{\rho, t}^{c} \times \Omega_{\rho, t}$ we have $G_{\rho}(y) > t \geq G_{\rho}(x)$, 
and hence $\varphi(x) = 0$ and $\varphi(y) = 1/t - 1/G_{\rho}(y)$. Thus,
\begin{align*}
(G_{\rho}(y) - G_{\rho}(x))(\varphi(y) - \varphi(x)) 
&= (G_{\rho}(y) - G_{\rho}(x)) \left( \frac{1}{t} - \frac{1}{G_{\rho}(y)} \right) \\
&\geq (G_{\rho}(y) - t) \left( \frac{1}{t} - \frac{1}{G_{\rho}(y)} \right).
\end{align*}
By using \Cref{lem:inequality2} again, we obtain
\begin{align*}
(G_{\rho}(y) - t) \left( \frac{1}{t} - \frac{1}{G_{\rho}(y)} \right)
&\geq \left( \log G_{\rho}(y) - \log t \right)^2 \\
&= \left( \log \left( \frac{G_{\rho}(y)}{t} \lor 1 \right) 
- \log \left( \frac{G_{\rho}(x)}{t} \lor 1 \right) \right)^2.
\end{align*}
A similar argument shows that \eqref{eq:ineq} holds true when $(x,y) \in \Omega_{\rho, t} \times \Omega_{\rho, t}^{c}$. Therefore, \eqref{eq:ineq} holds for all $x, y \in \Rd$.

We put the test function $\varphi$ into \eqref{eq:regularized-G}, and then 
use the inequality \eqref{eq:ineq}. Since $\varphi \leq 1/t$, we have
\begin{align*}
\frac{1}{t} 
&\geq \fint_{B_{\rho}(y_{0})} \varphi \dx
= \mathcal{E}(G_{\rho}, \varphi) \\
&\geq \int_{\Rd} \int_{\Rd} \left( \log \left( \frac{G_{\rho}(y)}{t} \lor 1 \right) 
- \log \left( \frac{G_{\rho}(x)}{t} \lor 1 \right) \right)^2 k(x,y) \dy \dx.
\end{align*}
By applying the assumption \eqref{assum:E-lower} on $u = \log( \frac{G_{\rho}}{t} \lor 1)$ we obtain
\begin{align*}
\frac{1}{t} 
\geq \cl (2-\alpha) \int_{\Rd} \int_{\Rd} \left( \log \left( \frac{G_{\rho}(y)}{t} \lor 1 \right) 
- \log \left( \frac{G_{\rho}(x)}{t} \lor 1 \right) \right)^2 |x-y|^{-d-\alpha} \dy \dx.
\end{align*}
Since a function $\log(G_{\rho}/t \lor 1)$ has a compact support in $\overline{\Omega}_{\rho, t}
\subset \overline{\Omega}$, we can apply \Cref{thm:Sobolev} (i) to obtain
\begin{align*}
\frac{1}{t} 
&\geq C \left( \int_{\Rd} \left| \log \left( \frac{G_{\rho}(x)}{t} \lor 1 \right)
\right|^{2d/(d-\alpha)} \dx \right)^{(d-\alpha)/d} \\ 
&\geq C \left( \int_{\Omega_{\rho, 2t}} \left| \log \left( \frac{G_{\rho}(x)}{t}
\right) \right|^{2d/(d-\alpha)} \dx \right)^{(d-\alpha)/d}
\geq C (\log 2)^2 | \Omega_{\rho, 2t} |^{(d-\alpha)/d}.
\end{align*}
Therefore, we conclude
\begin{align} \label{eq:G-weak}
[G_{\rho}]_{L^{d/(d-\alpha)}_\mathrm{weak}(\Rd)} = [G_{\rho}]_{L^{d/(d-\alpha)}_\mathrm{weak}(\Omega)}
= \sup_{t > 0} \, t | \Omega_{\rho, t} |^{(d-\alpha)/d} \leq C,
\end{align}
where $C$ depends on $d$, $\cl$, and $\alpha_0$ only. We remark that the uniform estimate \eqref{eq:G-weak} does not require any regularity of the boundary of $\Omega$.

Let us next prove the existence of a Green function of $L$ on $\Omega$. \Cref{prop:uniform_estimates} shows that for any $\beta \in (0,\alpha)$ and $q \in [1, d/(d-\alpha/2))$, $[G_{\rho}]_{W^{\beta/2,q}(\Rd)}$ is uniformly bounded. Moreover, we deduce from \eqref{eq:weak_Lq_embedding} and \eqref{eq:G-weak}
\begin{equation*}
\|G_{\rho}\|_{L^{q}(\Rd)} = \|G_{\rho}\|_{L^{q}(\Omega)} \leq C [G_{\rho}]_{L^{d/(d-\alpha)}_{\mathrm{weak}}(\Omega)} \leq C,
\end{equation*}
where $C$ depends on $d$, $\cl$, $\alpha_{0}$, and $|\Omega|$. Thus, $G_{\rho}$ is uniformly bounded in $W^{\beta/2,q}_{\Omega}(\Rd)$. By considering sequences $\rho_{i} \searrow 0$, $\beta_{i} \nearrow \alpha$, and $1< q_{i} \nearrow \frac{d}{d-\alpha/2}$, we find by a diagonal process a subsequence $G_{\rho_{i_{k}}}$ of $G_{\rho_{i}}$, which we denote by $G_{\rho_{k}}=G_{\rho_{i_{k}}}$, and a nonnegative function $G = G(\cdot, y_{0})$ such that
\begin{equation*}
G_{\rho_{k}} \rightharpoonup G \quad \text{in } \bigcap_{i=1}^{\infty} W^{\beta_{i}/2, q_{i}}_{\Omega}(\Rd).
\end{equation*}
In particular, we have
\begin{equation*}
G_{\rho_{k}} \rightharpoonup G \quad \text{in } W^{\beta/2, q}_{\Omega}(\Rd)
\end{equation*}
for all $\beta \in (0,\alpha)$ and $q \in [1, d/(d-\alpha/2))$ since we can find a large index $i$ so that $\beta_{i} > \beta$, $q_{i} > q$, and hence $W^{\beta_{i}/2, q_{i}}_{\Omega}(\Rd) \subset W^{\beta/2, q}_{\Omega}(\Rd)$ by \Cref{prop:cont-embedding}. Here, we used the following fact: if a linear map between two normed spaces is continuous with respect to their norms, then it is continuous with their respective weak topology. Note that we also have
\begin{equation} \label{eq:weak-L1}
G_{\rho_{k}} \rightharpoonup G \quad \text{in } L^{1}(\Omega)
\end{equation}
since $W^{\beta/2,1}_{\Omega}(\Rd) \subset L^{1}(\Omega)$, and we may assume $G_{\rho_{k}} \to G$ a.e. in $\Rd$.

We now prove that $G$ is a Green function of $L$ on $\Omega$. It is enough to check \eqref{eq:Green_function}. Let $\varphi \in H_{\Omega}(\Rd;k) \cap C(\Omega)$ and $\psi \in L^{\infty}(\Omega)$ satisfy \eqref{eq:DP}. Since $G_{\rho_{k}} \in H_{\Omega}(\Rd;k)$, we have from \eqref{eq:DP} and \eqref{eq:regularized-G}
\begin{equation} \label{eq:Gk}
\langle G_{\rho_{k}}(\cdot, y_{0}), \psi \rangle = \mathcal{E}(G_{\rho_{k}}(\cdot, y_{0}), \varphi) = \fint_{B_{\rho_{k}}(y_{0})} \varphi(x) \dx.
\end{equation}
Since $\psi \in L^{\infty}(\Omega)$, by \eqref{eq:weak-L1} the left-hand side of \eqref{eq:Gk} converges to $\langle G(\cdot, y_{0}), \psi \rangle$ as $k \to \infty$. Clearly, the right-hand side of \eqref{eq:Gk} converges to $\varphi(y_{0})$ as $k \to \infty$. Thus, the equality \eqref{eq:Green_function} follows.

To prove \eqref{eq:weak_Lp}, we set $\Omega_{t}=\lbrace x\in \Rd: G(x)>t \rbrace \subset \Omega$ for $t>0$. By \eqref{eq:weak_Lq_embedding} and \eqref{eq:G-weak} we have
\begin{equation*}
\|G_{\rho_{k}}\|_{L^{1}(\Omega_{t})} \leq \frac{d}{\alpha} |\Omega_{t}|^{\alpha/d} [G_{\rho_{k}}]_{L^{d/(d-\alpha)}_{\mathrm{weak}}(\Omega_{t})} \leq C \,|\Omega_{t}|^{\alpha/d}.
\end{equation*}
Since the $L^{1}$-norm is weakly lower semicontinuous, we obtain
\begin{equation*}
t |\Omega_{t}| \leq \|G\|_{L^{1}(\Omega_{t})} \leq \liminf_{k \to \infty} \|G_{\rho_{k}}\|_{L^{1}(\Omega_{t})} \leq C |\Omega_{t}|^{\alpha/d},
\end{equation*}
which implies
\begin{equation*}
[G]_{L^{d/(d-\alpha)}_{\mathrm{weak}}(\Rd)} = \sup_{t > 0} t |\Omega_{t}|^{(d-\alpha)/d} \leq C,
\end{equation*}
where $C$ depends only on $d$, $\cl$, and $\alpha_{0}$.

We next prove the uniqueness under the assumptions \eqref{assum:E-lower} and \eqref{assum:U1}. Assume that $G_{1}$ and $G_{2}$ are Green functions of $L$ on $\Omega$. Let $\psi \in L^{\infty}(\Omega)$. By the unique solvability of the Dirichlet problem \cite[Proposition 3.4]{FKV15}, there is a unique weak solution $\varphi \in H_{\Omega}(\Rd; k)$ of $-L\varphi=\psi$ in $\Omega$, i.e., $\mathcal{E}(\varphi, v)=\langle \psi, v \rangle$ for all $v \in H_{\Omega}(\Rd; k)$. The weak solution $\varphi$ is locally H\"older continuous in $\Omega$, see \cite[Theorem 1.1 (ii)]{KaWe22}, which covers more general nonsymmetric operators. Note that the assumption \eqref{assum:E-lower} implies (Poinc) and (Sob), and \eqref{assum:U1} implies (Cutoff) and ($\infty$-Tail). Since $\varphi$ is continuous, we conclude from the definition of the Green function
\begin{equation} \label{eq:G-unique}
\langle G_{1}(\cdot, y), \psi \rangle = \varphi(y) = \langle G_{2}(\cdot, y), \psi \rangle
\end{equation}
for every $y \in \Omega$. Since \eqref{eq:G-unique} holds for arbitrarily $\psi \in L^{\infty}(\Omega)$, we deduce $G_{1}(\cdot, y) = G_{2}(\cdot, y)$ a.e. in $\Omega$. Furthermore, since $G_{1}(\cdot, y)$ and $G_{2}(\cdot, y)$ are continuous on $\Omega \setminus \lbrace y \rbrace$, we conclude  $G_{1}(x, y) = G_{2}(x, y)$ for all $x, y \in \Omega$, $x\neq y$.
\end{proof}

\section{Pointwise upper bounds} \label{sec:upper-bounds}

This section is devoted to pointwise upper bounds for Green functions as stated in \Cref{thm:upper-bound}. The main idea is to prove a local boundedness result in \Cref{theo:local boundedness} involving a nonlocal tail. We control the tail by means of the uniform estimate \eqref{eq:weak_Lp}. The local boundedness with tail was first established in \cite{DCKP16} for weak solutions to nonlocal operators and made robust in \cite{Coz17}. It was further extended in \cite{Sch20} to more general operators with kernels satisfying \eqref{assum:E-lower} and \eqref{assum:U2}. The following uniform estimate is a simplified version of the result in \cite[Theorem 11.10]{Sch20}:
\begin{align} \label{eq:local_boundedness_2}
\sup_{B_{r/2}(x_0)} u 
\leq C \left( \fint_{B_r(x_0)} u_+^2(x) \dx \right)^{1/2} 
+ r^\alpha \int_{\Rd \setminus B_{r/2}(x_0)} \frac{u_+(y)}{|y-x_0|^{n+\alpha}} \dy.
\end{align}
However, this is not perfectly fit to Green functions because $L^2$-integrability for Green functions is not guaranteed. We need $L^q$-average with $q \in (1, d/(d-\alpha))$ instead of $L^2$-average in 
\eqref{eq:local_boundedness_2}. The local boundedness with $L^q$-average can be proved in the standard way as in \cite{DCKP16, Coz17, Sch20} by using Moser's iteration, but let us include the whole proof since it is nowhere written.

Let us first prove the following Caccioppoli-type estimates 
which will be used in the proof of the local boundedness. 

\begin{lemma} [Caccioppoli estimates] \label{lem:Caccioppoli}
Let $u \in H(\Rd; k)$ satisfy
\begin{align} \label{eq:subsolution}
\mathcal{E}(u, \varphi) \leq 0 \quad\text{for every nonnegative}~
\varphi \in H_{B_{\rho}(x_{0})}(\Rd; k).
\end{align}
For any $q > 1$ there exists a constant $C$, depending only on $q$, 
such that for any nonnegative function $\eta \in C_c^\infty(B_\rho(x_0))$
\begin{align*}
\int_{B_\rho(x_0)} \int_{B_\rho(x_0)} 
&\left( \eta(y) w^{q/2}(y) - \eta(x) w^{q/2}(x) \right)^2 k(x,y) \dy \dx \\
\leq& ~ C \int_{B_\rho(x_0)} \int_{B_\rho(x_0)} (w^q(y) + w^q(x))
|\eta(y) - \eta(x)|^2 k(x,y) \dy \dx \\
&+ C \int_{B_\rho(x_0)} \int_{\Rd \setminus B_\rho(x_0)} w(y) \eta^2(x) w^{q-1}(x) k(x,y) \dy \dx,
\end{align*}
where $w := (u-k)_+$ with $k \geq 0$.
\end{lemma}

In \Cref{lem:Caccioppoli} and \Cref{theo:local boundedness} we assume $u \in H(\Rd; k)$, which is sufficient for our purposes. Of course, one could weaken the assumption by assuming regularity of $u$ only around the point $x_0$.

\begin{proof}
In this proof, let $B = B_\rho(x_0)$. Let $\eta : \Rd \rightarrow [0,1]$
be a smooth function with $\mathrm{supp} \, \eta \subset B$. By putting $\varphi = \eta^{2} w^{q-1} \in H_{B}(\Rd; k)$ into the equation \eqref{eq:subsolution}, we have
\begin{equation} \label{eq:I_1,I_2_Caccioppoli}
\begin{split}
0 \geq&~ \int_B \int_B (u(y) - u(x))(\varphi(y) - \varphi(x)) k(x,y) \dy \dx \\
&+ 2 \int_B \int_{\Rd \setminus B} (u(y) - u(x))(-\varphi(x)) k(x,y) \dy \dx
=: I_1 + I_2.
\end{split}
\end{equation}
For $I_1$, we use an inequality
\begin{align*}
(u(y) - u(x))(\varphi(y) - \varphi(x))
\geq (w(y) - w(x)) \left( \eta^2 (y) w^{q-1}(y) - \eta^2(x) w^{q-1}(x) \right)
\end{align*}
and \Cref{lem:inequality3} to obtain
\begin{equation} \label{eq:I_1_Caccioppoli}
\begin{split}
I_1 
\geq&~ \frac{q-1}{32q^2} \int_B \int_B \left( \eta(y) 
w^{q/2}(y) - \eta(x) w^{q/2}(x) \right)^2 k(x,y) \dy \dx \\
&- 2 \left( 1 \lor \frac{1}{q-1} \right) \int_B \int_B 
(w^q(y) + w^q(x)) |\eta(y) - \eta(x)|^2 k(x,y) \dy \dx \,.
\end{split}
\end{equation}
For $I_2$, we observe that
\begin{align*}
(u(y) - u(x)) (-\eta^2(x) w^{q-1}(x)) \geq (u(y) - k)(-\eta^2(x) w^{q-1}(x)) 
\geq -w(y) \eta^2(x) w^{q-1}(x),
\end{align*}
from which we estimate
\begin{equation} \label{eq:I_2_Caccioppoli}
I_2 \geq - 2 \int_B \int_{\Rd \setminus B} w(y) \eta^2(x) w^{q-1}(x) k(x,y) \dy \dx.
\end{equation}
Combining \eqref{eq:I_1,I_2_Caccioppoli}, \eqref{eq:I_1_Caccioppoli}, and
\eqref{eq:I_2_Caccioppoli}, we conclude the lemma.
\end{proof}

We next prove the local boundedness with $L^q$-average and tail by using Moser's iteration technique. See also \cite{DCKP16, Coz17, Sch20}.

\begin{theorem} \label{theo:local boundedness}
Let $0 < \alpha_0 \leq \alpha < 2$, $\cl, \cK > 0$, and assume that $k$ satisfies \eqref{assum:E-lower} and \eqref{assum:U2}. For any $q > 1$, there exists a constant $C$, depending only on $d$, $q$, $\cl$, $\cK$, and $\alpha_0$, such that if $u \in H(\Rd; k)$ satisfies \eqref{eq:subsolution}, then 
\begin{align*}
\sup_{B_{r/2}(x_0)} u 
\leq C \left( \fint_{B_r(x_0)} u_+^q(x) \dx \right)^{1/q} 
+ r^\alpha \int_{\Rd \setminus B_{r/2}(x_0)} 
\frac{u_+(y)}{|y-x_0|^{d+\alpha}} \dy \,.
\end{align*}
\end{theorem}

\begin{proof}
For any $j = 0, 1, \cdots$, we define
\begin{align*}
&r_j = (1+2^{-j})\frac{r}{2}, \quad \tilde{r}_j = \frac{r_j+r_{j+1}}{2}, \quad 
B_j = B_{r_j}(x_0), \quad \tilde{B}_j = B_{\tilde{r}_j}(x_0), \\
&\eta_j \in C_c^\infty(\tilde{B}_j), \quad 0 \leq \eta_j \leq 1, \quad \eta_j \equiv 1
~\text{on}~ B_{j+1}, \quad |\nabla \eta_j| \leq 2^{j+3} / r, \\
&k_j = (1-2^{-j})K, \quad \tilde{k}_j = \frac{k_{j+1} + k_j}{2} \quad
\text{for some} ~ K \geq 0, ~ \text{which will be chosen later}, \\
&w_j = (u - k_j)_+, \quad \text{and} \quad \tilde{w}_j = (u - \tilde{k}_j)_+.
\end{align*}
By \Cref{lem:Caccioppoli} with $\rho = r_j, \eta = \eta_j$, and $k = \tilde{k}_j$, we obtain
\begin{align} \label{eq:I_1,I_2_LB}
\begin{split}
\int_{B_j} \int_{B_j} & \left( \eta_j(y) \tilde{w}_j^{q/2}(y) 
- \eta_j(x) \tilde{w}_j^{q/2}(x) \right)^2 k(x,y) \dy \dx \\
\leq&~ C \int_{B_j} \int_{B_j} \left( \tilde{w}_j^q(y) + \tilde{w}_j^q(x) \right) | \eta_j(y) 
- \eta_j(x) |^2 k(x,y) \dy \dx \\
&+ C \int_{B_j} \int_{\Rd \setminus B_j} \tilde{w}_j(y) \eta_j^2(x) \tilde{w}_j^{q-1}(x) 
k(x,y) \dy \dx =: I_1 + I_2.
\end{split}
\end{align}
Since $|\eta_j(y) - \eta_j(x)|^2 \leq 2^{2j+6} r^{-2} |y-x|^2$, the symmetry of $k$ and the assumption \eqref{assum:U2} yield
\begin{align} \label{eq:I_1_LB}
\begin{split}
I_1
&\leq C 2^{2j} r^{-2} (2-\alpha) \int_{B_j} \tilde{w}_j^q(x) \int_{B_j}  |y-x|^{2-d-\alpha} \dy \dx \\
&\leq C 2^{2j} r^{d-2} (2-\alpha) \fint_{B_j} w_j^q(x) \int_{B_j}  |y-x|^{2-d-\alpha} \dy \dx \\
&\leq C 2^{2j} r^{d-\alpha} \fint_{B_j} w_j^q(x) \dx,
\end{split}
\end{align}
where $C$ depends on $d$, $q$, $\cK$, and $\alpha_0$. Note that the assumption \eqref{assum:U1} is sufficient at this point for the estimate of $I_1$.

For $I_2$ we use the inequalities $\tilde{w}_j \leq w_0 = u_+$ and
\begin{align*}
\tilde{w}_j^{q-1} = (u - \tilde{k}_j)_+^{q-1} \leq \frac{(u-k_j)_+^q}{\tilde{k}_j - k_j}
= \frac{w_j^q}{\tilde{k}_j - k_j} \leq 4K^{-1} 2^j w_j^q,
\end{align*}
and the assumption \eqref{assum:U2} to obtain
\begin{align*}
I_2 \leq C \frac{2^j}{K} \int_{\tilde{B}_j} \int_{\Rd \setminus B_j}
\frac{u_+(y) w_j^q(x)}{|y-x|^{d+\alpha}} \dy \dx,
\end{align*}
where $C = C(q, \cK) > 0$. If $x\in \tilde{B}_j$ and $y \in \Rd \setminus B_j$, then
\begin{align*}
\frac{| y-x_0 |}{|y-x|} \leq 1+ \frac{| x-x_0 |}{|y-x|}
\leq 1+ \frac{\tilde{r}_j}{r_j - \tilde{r}_j} \leq C 2^j.
\end{align*}
Thus, we have
\begin{align} \label{eq:I_2_LB}
I_2 \leq C \frac{2^{j(d+\alpha+1)}}{K} r^{d-\alpha} \left( r^\alpha \int_{\Rd \setminus B_{r/2}(x_0)}
\frac{u_+(y)}{|y-x_0|^{d+\alpha}} \dy \right) \fint_{B_j} w_j^q(x) \dx.
\end{align}
On the other hand, applying the \Cref{thm:Sobolev} (ii)
to $\eta_j \tilde{w}_j^{q/2}$ leads us to
\begin{align} \label{eq:Sobolev}
\left( \int_{B_{j+1}} \tilde{w}_j^{q \chi}(x) \dx \right)^{1/\chi} \leq C \mathcal{E}^{\alpha}_{B_j}\left( \eta_j \tilde{w}_j^{q/2}, \eta_j \tilde{w}_j^{q/2} \right) + C r_j^{-\alpha} \int_{B_j} \tilde{w}_j^q(x) \dx,
\end{align}
where $\chi = d/(d-\alpha)$. Therefore, using \eqref{assum:E-lower} we 
combine inequalities \eqref{eq:I_1,I_2_LB}--\eqref{eq:Sobolev} as
\begin{align*}
\left( \fint_{B_{j+1}} \tilde{w}_j^{q \chi}(x) \dx \right)^{1/\chi} 
\leq C \left( 2^{2j} + \frac{2^{j(d+\alpha+1)}}{K} r^\alpha\int_{\Rd \setminus B_{r/2}(x_0)} \frac{u_+(y)}{| y - x_0 |^{d+\alpha}} \dy \right) \fint_{B_j} w_j^q(x) \dx.
\end{align*}
We set $A_j := (\fint_{B_j} w_j^q(x) \dx)^{1/q}$ and assume
\begin{align} \label{eq:K}
K \geq r^\alpha \int_{\Rd \setminus B_{r/2}(x_0)} 
\frac{u_+(y)}{| y - x_0 |^{d+\alpha}} \dy.
\end{align}
Then, since
\begin{align*}
\tilde{w}_j^{q \chi} 
= (u - \tilde{k}_j)_+^{q \chi} 
\geq (k_{j+1} - \tilde{k}_j)^{q (\chi - 1)} (u - k_{j+1})_+^q
= \left( \frac{K}{2^{j+2}} \right)^{q(\chi-1)} w_{j+1}^q,
\end{align*}
we obtain
\begin{align*}
\frac{A_{j+1}}{K} \leq \tilde{C} C_0^j \left( \frac{A_j}{K} \right)^\chi,
\end{align*}
where $\tilde{C} = 2^{2(\chi-1)} C^{\chi/q}$ 
and $C_0 = 2^{\frac{\chi(d+\alpha+1)}{q} -1+ \chi} > 1$. It will follow that 
$A_j \rightarrow 0$ as $j \rightarrow \infty$, provided that $A_0 \leq 
\tilde{C}^{-\frac{1}{\chi-1}} C_0^{-\frac{1}{(\chi-1)^2}} K$. Thus, we choose
\begin{align*}
K = 
\tilde{C}^\frac{1}{\chi-1} C_0^\frac{1}{(\chi-1)^2} A_0
+ r^\alpha \int_{\Rd \setminus B_{r/2}(x_0)} \frac{u_+(y)}{| y - x_0 |^{d+\alpha}} \dy,
\end{align*}
which is in accordance with \eqref{eq:K}. Note that since $(\chi-1)^{-1} \leq 1$ 
and $\chi/(\chi-1) = d/\alpha \leq d/\alpha_0$, we have
\begin{align*}
\tilde{C}^\frac{1}{\chi-1} C_0^\frac{1}{(\chi-1)^2} 
= 4C^\frac{\chi}{q(\chi-1)} 2^{\frac{\chi}{\chi-1} \frac{d+\alpha+1}{q} + \frac{1}{\chi-1}} 
\leq 4C^\frac{d}{q\alpha_0} 2^{\frac{d(d+3)}{q\alpha_0}+1}.
\end{align*}
Therefore, we conclude that
\begin{align*}
\sup_{B_{r/2}(x_0)} u 
\leq K \leq C \left( \fint_{B_r(x_0)} u_+^q(x) \dx \right)^{1/q} 
+ r^\alpha \int_{\Rd \setminus B_{r/2}(x_0)} 
\frac{u_+(y)}{| y - x_0 |^{d+\alpha}} \dy,
\end{align*}
where $C$ depends only on $d$, $q$, $\cl$, $\cK$ and $\alpha_0$.
\end{proof}

We are now in a position to prove the upper bound of Green functions 
using \Cref{theo:local boundedness} and the uniform estimates \eqref{eq:weak_Lp}. 

\begin{proof} [Proof of \Cref{thm:upper-bound}]
Let $G$ be the Green function of $L$ on $\Omega$, which is the limit of regularized Green functions $G_{\rho}$. Let $x_0, y_0 \in \Omega$ with $x_0 \neq y_0$ and assume $r := | x_0 - y_0 |/2 \geq \rho$. We first consider the case $B_r(x_0) \subset \Omega$. Since $\mathcal{E}(G_{\rho}(\cdot, y_0), \varphi) = 0$ for all $\varphi \in H_{B_{r}(x_{0})}(\Rd; k)$, \Cref{theo:local boundedness} shows that for any $q > 1$
\begin{align} \label{eq:I_1,I_2_Upper}
\sup_{B_{r/2}(x_0)} G_{\rho}(\cdot, y_0)
\leq C r^{-d/q} \|G_{\rho}(\cdot, y_0)\|_{L^q(B_r(x_0))} + r^\alpha \int_{\Rd \setminus B_{r/2}(x_0)} \frac{G_{\rho}(x, y_0)}{| x-x_0 |^{d+\alpha}} \dx
=: I_1 + I_2,
\end{align}
where $C$ depends only on $d$, $q$, $\cl$, $\cK$, and $\alpha_0$. Note that the essential supremum in the left-hand side of \eqref{eq:I_1,I_2_Upper} is
realized as the supremum by the interior regularity results \cite{KaWe22,DyKa20}. In particular, we have
\begin{align} \label{eq:sup}
G_{\rho}(x_0, y_0) \leq \sup_{B_{r/2}(x_0)} G_{\rho}(\cdot, y_0).
\end{align}
Let us choose $q = (1+d/(d-\alpha_0))/2 \in (1, d/(d-\alpha))$ 
so that the constant $C$ in \eqref{eq:I_1,I_2_Upper} depends on $d$, $\cl$, $\cK$, and $\alpha_0$.
Moreover, this choice of $q$ also makes the constant in the 
following estimates depend only on $d$, $\cl$, $\cK$, and $\alpha_0$: by the inequality \eqref{eq:weak_Lq_embedding} with $p=d/(d-\alpha)$ we have
\begin{align*}
I_1 \leq C \left( 1+ \frac{q}{p-q} \right)^{1/q} r^{\alpha-d} 
[G_{\rho}(\cdot, y_0)]_{L^{d/(d-\alpha)}_\mathrm{weak}(B_r(x_0))}.
\end{align*}
Using \eqref{eq:G-weak}, we obtain
\begin{align} \label{eq:I_1_Upper}
I_1 \leq C r^{\alpha-d},
\end{align}
where we used
\begin{align*}
p-q 
\geq \frac{d}{d-\alpha_0} - \frac{1}{2} \left( 1+\frac{d}{d-\alpha_0} \right) 
= \frac{1}{2} \frac{\alpha_0}{d-\alpha_0}.
\end{align*}

To estimate $I_2$, we split it into two integrals:
\begin{align}
\begin{split} \label{eq:I_2_Upper}
I_2 
&= r^\alpha \int_{B_{r/2}^c (x_0) \cap \lbrace G \leq r^{\alpha-d} \rbrace} 
\frac{G_{\rho}(x, y_0)}{| x-x_0 |^{d+\alpha}} \dx 
+ r^\alpha \int_{B_{r/2}^c (x_0) \cap \lbrace G_{\rho} > r^{\alpha-d} \rbrace} 
\frac{G_{\rho}(x, y_0)}{| x-x_0 |^{d+\alpha}} \dx \\
&=: I_{2,1} + I_{2,2}.
\end{split}
\end{align}
The term $I_{2,1}$ can be easily computed as
\begin{align} \label{eq:I_2,1_Upper}
I_{2,1} 
\leq r^{2\alpha-d} \int_{\Rd \setminus B_{r/2}(x_0)} \frac{\dx} {| x-x_0 |^{d+\alpha}} 
= \frac{|\mathbb{S}^{d-1}|}{\alpha} r^{2\alpha-d} \left( \frac{r}{2} \right)^{-\alpha} 
\leq \frac{4|\mathbb{S}^{d-1}|}{\alpha_0} r^{\alpha-d}.
\end{align}
For $I_{2,2}$ we have
\begin{align*}
I_{2,2} 
\leq 2^{d+\alpha} r^{-d} \int_{\lbrace G_{\rho} > r^{\alpha-d} \rbrace} G_{\rho}(x, y_0) \dx 
\leq 2^{d+2} r^{-d} \int_{\Rd} \chi_{\lbrace G_{\rho} > r^{\alpha-d} \rbrace} 
\int_0^\infty \chi_{\lbrace G_{\rho} > t \rbrace} \dt \dx.
\end{align*}
By utilizing the Fubini Theorem we obtain
\begin{align*}
I_{2,2} 
\leq Cr^{-d} \int_0^\infty \left| \lbrace G_{\rho}(\cdot, y_0) > (r^{\alpha-d} \lor t) \rbrace \right| \dt.
\end{align*}
We now make use of the estimate \eqref{eq:G-weak} and deduce
\begin{equation} \label{eq:I_2,2_Upper}
I_{2,2} \leq C r^{-d} [G_{\rho}(\cdot, y_0)]_{L^{d/(d-\alpha)}_\mathrm{weak}
(\Omega)}^{d/(d-\alpha)} \int_0^\infty (r^{\alpha-d} \lor t)^{-\frac{d}{d-\alpha}} \dt \leq C r^{\alpha-d},
\end{equation}
where $C = C(d, \cl, \alpha_0) > 0$. By combining \eqref{eq:I_1,I_2_Upper}--\eqref{eq:I_2,2_Upper}, we conclude that
\begin{align*}
G_{\rho}(x_0, y_0) \leq C |x_0-y_0|^{\alpha-d},
\end{align*}
where $C = C(d, \cl, \cK, \alpha_0) > 0$, from which \eqref{eq:upper-bound} follows by taking $\rho \to 0$.

Finally, let us consider the case $B_r(x_0) \not\subset \Omega$. In this case we 
consider a bounded open set $\tilde{\Omega} \supset \Omega$ such that $B_r(x_0) \subset \tilde{\Omega}$ and let $\tilde{G}$ be the Green function 
of $L$ on $\tilde{\Omega}$, which is defined as the limit of regularized Green functions $\tilde{G}_{\rho}$. Then 
\begin{align*}
\mathcal{E}\left(G_{\rho}(\cdot, y_0) - \tilde{G}_{\rho}(\cdot, y_0), \varphi \right) = 0 
\quad\text{for all} ~ \varphi \in H_{\Omega}(\Rd;k).
\end{align*}
Since $G_{\rho}(\cdot, y_0) = 0$ $\mathrm{a.e.}$ on $\Rd \setminus \Omega$, we have $G_{\rho}(\cdot, y_0) \leq \tilde{G}_{\rho}(\cdot, y_0)$ $\mathrm{a.e.}$ on $\Rd \setminus \Omega$. 
Thus, using $\varphi := (G_{\rho}(\cdot, y_0) - \tilde{G}_{\rho}(\cdot, y_0))_+ \in H_{\Omega}(\Rd; k)$ as a test function we obtain
\begin{align*}
0 
= \mathcal{E} \left( G_{\rho}(\cdot, y_0) - \tilde{G}_{\rho}(\cdot, y_0), \varphi \right) 
\geq \mathcal{E}(\varphi, \varphi) 
\geq \cl \mathcal{E}^{\alpha}(\varphi, \varphi) 
\geq 0.
\end{align*} 
Here we used the assumption \eqref{assum:E-lower}. Therefore, we have $\varphi = 0$ $\mathrm{a.e.}$ in 
$\Omega$, which in turn implies $G_{\rho}(\cdot, y_0) \leq \tilde{G}_{\rho}(\cdot, y_0)$ 
$\mathrm{a.e.}$ in $\Omega$. Since $G_{\rho}(\cdot, y_0) - \tilde{G}_{\rho}(\cdot, y_0)$ is H\"older 
continuous in $\Omega \setminus \lbrace y_0 \rbrace$, we have $G_{\rho}(\cdot, y_0) 
\leq \tilde{G}_{\rho}(\cdot, y_0)$ in $\Omega \setminus \lbrace y_0 \rbrace$. 
The upper bound of $G_{\rho}$ follows from the upper bound of $\tilde{G}_{\rho}$.
\end{proof}

\section{Pointwise lower bounds} \label{sec:lower-bounds}

The aim of this section is to prove the pointwise lower bounds for the Green function by modifying the classical proof for second order differential operators in \cite{GrWi82}. The main tool for this proof (\cite[Equation (1.9)]{GrWi82}) is the Harnack inequality together with a localization technique that cuts out the singularity of the Green function. We use similar cut-off functions for nonlocal operators, but the results are very different because the localization technique produces not only local quantities corresponding to those appearing in \cite{GrWi82} but also nonlocal quantities. Therefore, the following lemmas focus on estimating nonlocal quantities.

We begin with an estimate of a double integral of local-nonlocal nature.
This quantity can be made small by assuming the 
local region to be very small.

\begin{lemma} \label{lem:small_term}
Let $0 < \alpha_0 \leq \alpha < 2$, $\cl, \cU > 0$, and assume that $k$ satisfies \eqref{assum:E-lower} and \eqref{assum:U1}. Let $y_0 \in \Omega$. There exists a constant 
$\varepsilon < 1/2$, depending only on $d$, $\cl$, $\cU$, and $\alpha_0$, such that 
\begin{align*}
\int_{B_{\varepsilon r}(y_0)} \int_{\Rd \setminus B_r(y_0)} 
G_{\rho}(x, y_0) k(x,y) \dy \dx \leq \frac{1}{4}
\end{align*}
for all $r < \mathrm{dist}(y_0, \partial \Omega)$.
\end{lemma}

\begin{proof}
Let us denote $B_r = B_r(y_0)$.
If $x \in B_{\varepsilon r}$ and $y \in \Rd \setminus B_r$, then 
$|y-x| \geq | y-y_0 | - | x - y_0 | \geq r (1 - \varepsilon)$.
Thus, using the assumption \eqref{assum:U1} we obtain
\begin{align} \label{eq:I}
\begin{split}
\int_{B_{\varepsilon r}} \int_{\Rd \setminus B_r} G_{\rho}(x,y_0) k(x,y) \dy \dx 
&\leq \int_{B_{\varepsilon r}} G_{\rho}(x, y_0) \int_{\Rd \setminus B(x, (1-\varepsilon)r)} k(x,y) \dy \dx \\
&\leq \frac{\cU}{(1- \varepsilon)^{\alpha} r^{\alpha}} \| G_{\rho}(\cdot, y_0) \|_{L^1 (B_{\varepsilon r})}.
\end{split}
\end{align}
We utilize the inequality \eqref{eq:weak_Lq_embedding} 
with $q = 1$ and $p = d/(d-\alpha)$ and the estimate \eqref{eq:G-weak} to get
\begin{align} \label{eq:L1_norm}
\|G_{\rho}(\cdot, y_0)\|_{L^1(B_{\varepsilon r})} 
\leq \frac{d}{\alpha} | B_{\varepsilon r} |^{\alpha/d} 
[G_{\rho}(\cdot, y_0)]_{L^{d/(d-\alpha)}_\mathrm{weak}(B_{\varepsilon r})}
\leq \frac{d}{\alpha_0} |B_1| (\varepsilon r)^\alpha.
\end{align}
Combining \eqref{eq:I} and \eqref{eq:L1_norm}, we have
\begin{align*}
\int_{B_{\varepsilon r}} \int_{\Rd \setminus B_r} G_{\rho}(x, y_0) k(x,y) \dy \dx 
\leq C \left( \frac{\varepsilon}{1-\varepsilon} \right)^\alpha,
\end{align*}
where $C$ depends only on $d$, $\cl$, $\cU$, and $\alpha_0$. 
Note that the assumption \eqref{assum:E-lower} was used in the estimate \eqref{eq:weak_Lp}. We take
\begin{align*}
\varepsilon < \frac{1}{2} \min \lbrace 1, (4C)^{-1/\alpha_0} \rbrace
\end{align*}
so that $C(\varepsilon / (1-\varepsilon))^\alpha
\leq C (2\varepsilon)^{\alpha_0} < 1/4$, which finishes the proof.
\end{proof}

The next lemma shows how the integral over a global region can be controlled by 
a local quantity. The method used in the following lemma is inspired by 
\cite[Lemma 4.2]{DCKP14}. The difference is that we use a cut-off function 
whose support is in an annulus near the singularity of Green functions. 
More precisely, we use a cut-off function $\eta : \Rd \rightarrow \R$ satisfying 
\begin{align} \label{eq:cut-off_function}
\eta \in [0,1], \quad 
\eta = 1~\text{in}~ A^r_{\varepsilon r}, \quad 
\eta = 0 ~\text{in}~ \Rd \setminus A^{3r/2}_{\varepsilon r/2}, \quad\text{and}\quad 
| \nabla \eta | \leq \frac{4}{\varepsilon} r^{-1},
\end{align}
where $A^R_r$ denotes an annulus $B_R(y_0) \setminus B_r(y_0)$ and $\varepsilon$ is the constant from \Cref{lem:small_term}.

\begin{lemma} \label{lem:global_term}
Let $0 < \alpha_0 \leq \alpha < 2$, $\cl, \cU > 0$, and assume that $k$ satisfies \eqref{assum:E-lower} and \eqref{assum:U1}. Let $y_0 \in \Omega$ and let $\eta$ be the cut-off function satisfying \eqref{eq:cut-off_function}. There exists a constant $C$, depending only on $d$, $\cl$, $\cU$, and $\alpha_0$, such that
\begin{align} \label{eq:global_term_eta}
\int_{\Rd \setminus A^{3r/2}_{\varepsilon r/2}(y_0)}
\int_{A^{3r/2}_{\varepsilon r/2}(y_0)}G_{\rho}(x, y_0) \eta^2(y) k(x,y) \dy \dx
\leq C r^{d-\alpha} \sup_{A^{3r/2}_{\varepsilon r/2}(y_0)} G_{\rho}(\cdot, y_0)
\end{align}
for all $r < \mathrm{dist}(y_0, \partial \Omega)/2$ and $\rho < \varepsilon r/2$. In particular, 
\begin{align} \label{eq:global_term}
\int_{\Rd \setminus A^{3r/2}_{\varepsilon r/2}(y_0)}
\int_{A^r_{\varepsilon r}(y_0)}G_{\rho}(x, y_0) k(x,y) \dy \dx
\leq C r^{d-\alpha} \sup_{A^{3r/2}_{\varepsilon r/2}(y_0)} G_{\rho}(\cdot, y_0).
\end{align}
\end{lemma}

We will use the following estimate in the proof of \Cref{lem:global_term}.

\begin{lemma} \label{lem:cut-off_function}
Under the same setting as in \Cref{lem:global_term},
\begin{align} \label{eq:cut-off}
\int_{B_{3r/2}(y_0)} \int_{\Rd} |\eta(y) - \eta(x)|^2 k(x,y) \dy \dx \leq C r^{d-\alpha}
\end{align}
for some $C$ depending only on $d$, $\cl$, $\cU$, and $\alpha_0$.
\end{lemma}

\begin{proof}
It follows from \eqref{assum:U1} that
\begin{equation*}
\int_{\Rd} | \eta(y) - \eta(x) |^2 k(x,y) \dy \leq \int_{\Rd} \left( 1 \land \frac{16}{\varepsilon^2 r^2} |y-x|^2 \right) k(x,y) \dy \leq Cr^{-\alpha},
\end{equation*}
where $C$ depends on $\varepsilon$ and $\cU$. Therefore, the estimate \eqref{eq:cut-off} holds with a constant $C$ depending on $d$, $\cl$, $\cU$, and $\alpha_0$.
\end{proof}

We are now in a position to prove \Cref{lem:global_term} using \Cref{lem:cut-off_function}.

\begin{proof} [Proof of \Cref{lem:global_term}]
Let us write $G_{\rho} = G_{\rho}(\cdot, y_0)$, $A^R_r = B_R(y_0) \setminus B_r(y_0)$, and set
\begin{align*}
k = \sup_{A^{3r/2}_{\varepsilon r/2}} G_{\rho}.
\end{align*}
Since $G_{\rho}$ is continuous in $\Omega$, we have $k < \infty$. We put $\varphi = (G_{\rho} - 2k) \eta^2 \in H_{\Omega}(\Rd; k)$ into \eqref{eq:regularized-G} to obtain
\begin{align} \label{eq:I_1,I_2_global}
\begin{split}
0 =&~ \int_{A^{3r/2}_{\varepsilon r/2}} \int_{A^{3r/2}_{\varepsilon r/2}}
(G_{\rho}(y) - G_{\rho}(x))(\varphi(y) - \varphi(x)) k(x,y) \dy \dx \\
&+ 2\int_{\Rd \setminus A^{3r/2}_{\varepsilon r/2}} 
\int_{A^{3r/2}_{\varepsilon r/2}} (G_{\rho}(y) - G_{\rho}(x))\varphi(y) k(x,y) \dy \dx 
=: I_1+ I_2.
\end{split}
\end{align}
Let $w = G_{\rho} - 2k$. For $x,y \in A^{3r/2}_{\varepsilon r/2}$ 
with $\eta(y) \geq \eta(x)$, we have
\begin{align*}
(G_{\rho}(y) - G_{\rho}(x))(\varphi(y) - \varphi(x)) 
&= (w(y) - w(x))^2 \eta^2(y) + (w(y) - w(x)) w(x) (\eta^2(y) - \eta^2(x)) \\
&\geq (w(y) - w(x))^2 \eta^2(y) - 2 | w(y) - w(x) |
| w(x) | \eta(y) | \eta(y) - \eta(x) | \\
&\geq - | w(x) |^2 | \eta(y) - \eta(x) |^2 \geq - 4k^2 | \eta(y) - \eta(x) |^2.
\end{align*}
Note that the resulting inequality remains true when $\eta(x) \leq \eta(y)$. This inequality and \Cref{lem:cut-off_function} yield that
\begin{align} \label{eq:I_1_global}
I_1 
\geq - 4k^2 \int_{A^{3r/2}_{\varepsilon r/2}} \int_{A^{3r/2}_{\varepsilon r/2}}
| \eta(y) - \eta(x) |^2 k(x,y) \dy \dx 
\geq -Ck^2 r^{d-\alpha},
\end{align}
where $C$ depends only on $d$, $\cl$, $\cU$, and $\alpha_0$.

For $I_2$, we split the integral into two parts as
\begin{align*}
I_2
=&~ 2\int_{\Rd \setminus A^{3r/2}_{\varepsilon r/2}} \int_{A^{3r/2}_{\varepsilon r/2}}
(G_{\rho}(y) - G_{\rho}(x))(G_{\rho}(y)-2k) \chi_{\lbrace G_{\rho}(x) \geq k \rbrace} \eta^2(y) k(x,y) \dy \dx \\
&+ 2\int_{\Rd \setminus A^{3r/2}_{\varepsilon r/2}} \int_{A^{3r/2}_{\varepsilon r/2}}
(G_{\rho}(y) - G_{\rho}(x))(G_{\rho}(y)-2k) \chi_{\lbrace G_{\rho}(x) < k \rbrace} \eta^2(y) k(x,y) \dy \dx 
=: I_{2,1} + I_{2,2}.
\end{align*}
Since $G_{\rho}(y) \leq k$ in $A^{3r/2}_{\varepsilon r/2}$, we have
\begin{align} \label{eq:integrand1}
(G_{\rho}(y) - G_{\rho}(x))(G_{\rho}(y)-2k) \chi_{\lbrace G_{\rho}(x) \geq k \rbrace} 
&= (G_{\rho}(x) - G_{\rho}(y))(2k - G_{\rho}(y)) \chi_{\lbrace G_{\rho}(x) \geq k \rbrace} \\
&\geq (G_{\rho}(x) - k)k
\end{align}
and
\begin{align} \label{eq:integrand2}
\begin{split}
(G_{\rho}(y) - G_{\rho}(x))(G_{\rho}(y)-2k) \chi_{\lbrace G_{\rho}(x) < k \rbrace}
&= -(G_{\rho}(y) - G_{\rho}(x)) (2k - G_{\rho}(y)) \chi_{\lbrace G_{\rho}(x) < k \rbrace} \\
&\geq -2k(G_{\rho}(y) - G_{\rho}(x))_+ \chi_{\lbrace G_{\rho}(x) < k \rbrace} \geq -2k^2.
\end{split}
\end{align}
Using \eqref{eq:integrand1} and \eqref{eq:integrand2}, we obtain
\begin{align} \label{eq:I_2,1_global}
I_{2,1} 
\geq 2k\int_{\Rd \setminus A^{3r/2}_{\varepsilon r/2}} 
\int_{A^{3r/2}_{\varepsilon r/2}} G_{\rho}(x) \eta^2(y) k(x,y) \dy \dx 
- 2k^2 \int_{\Rd \setminus A^{3r/2}_{\varepsilon r/2}} 
\int_{A^{3r/2}_{\varepsilon r/2}} \eta^2(y) k(x,y) \dy \dx
\end{align}
and
\begin{align} \label{eq:I_2,2_global}
I_{2,2} \geq -4k^2 \int_{\Rd \setminus A^{3r/2}_{\varepsilon r/2}} 
\int_{A^{3r/2}_{\varepsilon r/2}} \eta^2(y) k(x,y) \dy \dx,
\end{align}
respectively. 
We combine the estimates \eqref{eq:I_2,1_global} and \eqref{eq:I_2,2_global}, 
and then use \Cref{lem:cut-off_function} to estimate
\begin{align} \label{eq:I_2_global}
I_2
\geq 2k\int_{\Rd \setminus A^{3r/2}_{\varepsilon r/2}} 
\int_{A^{3r/2}_{\varepsilon r/2}} G_{\rho}(x) \eta^2(y) k(x,y) \dy \dx 
- Ck^2 r^{d-\alpha}.
\end{align}
The inequality \eqref{eq:global_term_eta} is established by combining
\eqref{eq:I_1,I_2_global}, \eqref{eq:I_1_global}, and \eqref{eq:I_2_global}. 
The second assertion \eqref{eq:global_term} follows immediately from 
\eqref{eq:cut-off_function}.
\end{proof}

The next lemma corresponds to the estimate of $L^2$-norm of the gradient of 
Green function in the case of second order differential operators. Global 
terms arising from the weak formulation of Green function can be 
controlled by using \Cref{lem:global_term}.

\begin{lemma} \label{lem:local_term}
Let $0 < \alpha_0 \leq \alpha < 2$, $\cl, \cU > 0$, and assume that $k$ satisfies \eqref{assum:E-lower} and \eqref{assum:U1}. Let $y_0 \in \Omega$ and let $\varepsilon$ be the constant in \Cref{lem:small_term}. There exists a constant $C$, depending only on $d$, $\cl$, $\cU$, and $\alpha_0$, such that
\begin{align*}
\int_{A^{3r/2}_{\varepsilon r/2}(y_0)} \int_{A^r_{\varepsilon r}(y_0)}
(G_{\rho}(y, y_0) - G_{\rho}(x, y_0))^2 k(x,y) \dy \dx
\leq C r^{d-\alpha} \sup_{A^{3r/2}_{\varepsilon r/2}(y_0)} G_{\rho}^{2}(\cdot, y_0).
\end{align*}
for all $r < \mathrm{dist}(y_0, \partial \Omega)/2$ and $\rho < \varepsilon r/2$.
\end{lemma}

\begin{proof}
As in the previous proof, let us write $G_{\rho} = G_{\rho}(\cdot, y_0)$, $A^R_r = B_R(y_0) \setminus B_r(y_0)$.
Let $\eta : \Rd \rightarrow \R$ be a cut-off function satisfying 
\eqref{eq:cut-off_function} and define $\varphi = G_{\rho}\eta^2 \in H_{\Omega}(\Rd; k)$. Then, we have from \eqref{eq:regularized-G}
\begin{align} \label{eq:I_1,I_2_local}
\begin{split}
0 
= \mathcal{E}(G_{\rho}, \varphi) 
=&~ \int_{A^{3r/2}_{\varepsilon r/2}} \int_{A^{3r/2}_{\varepsilon r/2}} 
(G_{\rho}(y) - G_{\rho}(x))(\varphi(y) - \varphi(x)) k(x,y) \dy \dx \\
&+ 2 \int_{\Rd \setminus A^{3r/2}_{\varepsilon r/2}} 
\int_{A^{3r/2}_{\varepsilon r/2}} (G_{\rho}(y) - G_{\rho}(x)) \varphi(y) k(x,y) \dy \dx 
=: I_1 + I_2.
\end{split}
\end{align}
We utilize \Cref{lem:inequality4} and \Cref{lem:cut-off_function} to estimate $I_1$ as
\begin{align} \label{eq:I_1_local}
\begin{split}
I_1 
\geq&~ \frac{1}{4} \int_{A^{3r/2}_{\varepsilon r/2}} \int_{A^{3r/2}_{\varepsilon r/2}} 
(G_{\rho}(y) - G_{\rho}(x))^2 (\eta^2(y) + \eta^2(x)) k(x,y) \dy \dx \\
&- 4 \int_{A^{3r/2}_{\varepsilon r/2}} \int_{A^{3r/2}_{\varepsilon r/2}} 
(G_{\rho}^{2}(y) + G_{\rho}^{2}(x)) | \eta(y) - \eta(x) |^2 k(x,y) \dy \dx \\
\geq&~ \frac{1}{4} \int_{A^{3r/2}_{\varepsilon r/2}} \int_{A^r_{\varepsilon r}} 
(G_{\rho}(y) - G_{\rho}(x))^2 k(x,y) \dy \dx - C r^{d-\alpha} \sup_{A^{3r/2}_{\varepsilon r/2}} G_{\rho}^{2}.
\end{split}
\end{align}
For $I_2$, we use \eqref{eq:global_term_eta} to obtain
\begin{align} \label{eq:I_2_local}
\begin{split}
I_2 
&\geq -2 \int_{\Rd \setminus A^{3r/2}_{\varepsilon r/2}} 
\int_{A^{3r/2}_{\varepsilon r/2}} G_{\rho}(x) G_{\rho}(y) \eta^2(y) k(x,y) \dy \dx \\
&\geq -2 \bigg( \sup_{A^{3r/2}_{\varepsilon r/2}} G_{\rho} \bigg) 
\int_{\Rd \setminus A^{3r/2}_{\varepsilon r/2}} \int_{A^{3r/2}_{\varepsilon r/2}} 
G_{\rho}(x) \eta^2(y) k(x,y) \dy \dx \\
&\geq -C r^{d-\alpha} \sup_{A^{3r/2}_{\varepsilon r/2}} G_{\rho}^{2}.
\end{split}
\end{align}
The proof is finished by combining \eqref{eq:I_1,I_2_local}--\eqref{eq:I_2_local}.
\end{proof}

The Harnack inequality \eqref{assum:H} for Green function on the annulus $B_{2r}(y) \setminus B_r(y)$ implies the same inequalities on larger annuli by standard covering argument.

\begin{lemma} \label{lem:Harnack}
Condition \eqref{assum:H} implies the following condition: For every $M >2$ there is $c=c(M, \cH) >0$ such that for every ball $B_{Mr}(y) \Subset \Omega$ 
\begin{align*}
\sup_{B_{Mr}(y) \setminus B_r(y)} G(\cdot, y) 
\leq c \inf_{B_{Mr}(y) \setminus B_r(y)} G(\cdot, y).
\end{align*}
\end{lemma}

\begin{proof}
We write $B_R = B_R(y)$. Note that the infimum of $G(\cdot, y)$ over 
$B_{Mr} \setminus B_r$ is attained in some smaller annulus, say
\begin{equation} \label{eq:inf_G}
\inf_{B_{R+r/4} \setminus B_{R}} G(\cdot, y) 
= \inf_{B_{Mr} \setminus B_r} G(\cdot, y)
\end{equation}
for some $B_{R+r/4} \setminus B_{R} \subset B_{Mr} \setminus B_r$.

Let us now cover the annulus $B_{Mr} \setminus B_r$ with 
overlapping annuli to use the chaining argument. 
There exist an integer $N = N(M) \in \mathbb{N}$ and
a sequence of radii $r = r_0 < \cdots < r_N = Mr/2$ such that 
$2r_{j-1} - r_{j} > r/2$ for all $j = 1, \dots, N$. 
We define $A_j = B_{2r_j} \setminus B_{r_j}$, then the family of 
annuli $\lbrace A_j \rbrace_{j=0}^N$ satisfies the following properties:
\begin{equation*}
\begin{split}
&B_{Mr} \setminus B_r = \cup_{j=0}^N A_j \quad\text{and} \\
&B_{r_j+3r/8} \setminus B_{r_j+r/8} \subset A_{j-1} \cap A_j
\quad\text{for}~ j=1, \dots, N.
\end{split}
\end{equation*}
In other words, the annulus $B_{Mr} \setminus B_{r}$ is covered by $\lbrace A_j \rbrace_{j=0}^N$ and each intersection $A_{j-1} \cap A_{j}$ between adjacent annuli has enough space to contain an annulus whose radius difference is $r/4$. Therefore, $B_R \setminus B_{R+r/4} \subset A_{j_0}$ for some 
$j_0 \in \lbrace 0, \dots, N\rbrace$.

For any $j = 0, \dots, N-1$, we observe that
\begin{equation*}
\sup_{A_j} G \leq c_H \inf_{A_j} G \leq c_H \inf_{A_j \cap A_{j+1}} G \leq c_H \sup_{A_j \cap A_{j+1}} G \leq c_H \sup_{A_{j+1}} G
\end{equation*}
by \eqref{assum:H}. Similarly, we also have $\sup_{A_{j+1}} G \leq c_H \sup_{A_j}G$.
Therefore, we have by \eqref{eq:inf_G}
\begin{equation*}
\sup_{A_j} G \leq c_H^{|j-j_0|} \sup_{A_{j_0}} G \leq c_H^{N+1} \inf_{A_{j_0}} G \leq c_H^{N+1} \inf_{B_{R+r/4} \setminus B_{R}} G = c_H^{N+1} \inf_{B_{Mr} \setminus B_r} G
\end{equation*}
for any $j=0,\dots, N$. We conclude the lemma by observing that
\begin{equation*}
\sup_{B_{Mr} \setminus B_r} G \leq \max_j \bigg\lbrace \sup_{A_j} G \bigg\rbrace \leq c_H^{N+1} \inf_{B_{Mr} \setminus B_r} G.
\end{equation*}
\end{proof}

We now provide the proof of pointwise lower bounds of Green functions 
by gathering pieces of integrals in the preceding lemmas.

\begin{proof} [Proof of \Cref{thm:lower-bound}]
Let $x_0, y_0 \in \Omega$ with $x_0 \neq y_0$ and let $r = |x_0 - y_0| \leq \mathrm{dist}(y_0, \partial \Omega)/2$. Let $\varepsilon$ be the constant in \Cref{lem:small_term} and let $\eta : \Rd \rightarrow \R$ be a cut-off function satisfying
\begin{align*}
\eta \in [0,1], \quad 
\eta = 1~\text{in}~ B_{\varepsilon r}(y_0), \quad 
\eta = 0 ~\text{outside}~ B_r(y_0), \quad\text{and}\quad
|\nabla \eta| \leq 4r^{-1}.
\end{align*}
Let us write $G_{\rho} = G_{\rho}(\cdot, y_0)$ with $\rho < \varepsilon r/2$ and $A^R_r = B_R \setminus B_r = B_R(y_0) \setminus B_r(y_0)$. By testing the equation \eqref{eq:regularized-G} with $\eta$, we have
\begin{align} \label{eq:I_1,I_2,I_3,I_4_lower}
\begin{split}
1
=&~ 2\int_{B_{\varepsilon r}} \int_{\Rd \setminus B_r} 
(G_{\rho}(y) - G_{\rho}(x)) (-1) k(x,y) \dy \dx \\
&+ 2\int_{B_{\varepsilon r/2}} \int_{A^r_{\varepsilon r}} 
(G_{\rho}(y) - G_{\rho}(x)) (\eta(y) - 1) k(x,y) \dy \dx \\
&+ 2\int_{\Rd \setminus B_{3r/2}} \int_{A^r_{\varepsilon r}} 
(G_{\rho}(y) - G_{\rho}(x)) \eta(y) k(x,y) \dy \dx \\
&+ \iint_{\left( A^{3r/2}_{\varepsilon r/2} \times A^r_{\varepsilon r} \right) 
\cup \left( A^r_{\varepsilon r} \times A^{3r/2}_{\varepsilon r/2} \right)} 
(G_{\rho}(y) - G_{\rho}(x)) (\eta(y) - \eta(x)) k(x,y) \dy \dx 
=: I_1 + I_2 + I_3 + I_4.
\end{split}
\end{align}

We first use \Cref{lem:small_term} to have
\begin{align} \label{eq:I_1_lower}
I_1 
\leq 2\int_{B_{\varepsilon r}} \int_{\Rd \setminus B_r} G_{\rho}(x) k(x,y) \dy \dx 
\leq \frac{1}{2}.
\end{align}

For $I_2$, we observe that 
$(G_{\rho}(y) - G_{\rho}(x))(\eta(y) - 1) = G_{\rho}(x)(1-\eta(y)) - G_{\rho}(y)(1-\eta(y)) \leq G_{\rho}(x)$. 
Thus, by means of \eqref{eq:global_term} we obtain
\begin{align} \label{eq:I_2_lower}
I_2 
\leq 2 \int_{B_{\varepsilon r/2}} \int_{A^r_{\varepsilon r}} G_{\rho}(x) k(x,y) \dy \dx
\leq Cr^{d-\alpha} \sup_{A^{3r/2}_{\varepsilon r/2}} G_{\rho}.
\end{align}

We next estimate the third term. Note that we have
\begin{align*}
I_3 
\leq 2 \int_{\Rd \setminus B_{3r/2}} \int_{A^r_{\varepsilon r}} G_{\rho}(y) k(x,y) \dy \dx
\leq 2\bigg( \sup_{A^r_{\varepsilon r}} G_{\rho} \bigg)
\int_{A^r_{\varepsilon r}} \int_{\Rd \setminus B_{3r/2}} k(x,y) \dy \dx,
\end{align*}
where we used the symmetry of $k$ in the last inequality. 
By the assumption \eqref{assum:U1} we obtain
\begin{align*}
\int_{A^r_{\varepsilon r}} \int_{\Rd \setminus B_{3r/2}} k(x,y) \dy \dx
\leq \int_{A^r_{\varepsilon r}} \int_{\Rd \setminus B(x, r/2)} k(x,y) \dy \dx
\leq Cr^{d-\alpha},
\end{align*}
from which we deduce
\begin{align} \label{eq:I_3_lower}
I_3 \leq C r^{d-\alpha} \sup_{A^r_{\varepsilon r}} G_{\rho}.
\end{align}

For the last term we make use of the symmetry of $k$ and H\"older's inequality to have
\begin{align*}
I_4 
&\leq 2\int_{A^{3r/2}_{\varepsilon r/2}} \int_{A^r_{\varepsilon r}}
| G_{\rho}(y) - G_{\rho}(x) | | \eta(y) - \eta(x) | k(x,y) \dy \dx \\
&\leq 2 \left( \int_{A^{3r/2}_{\varepsilon r/2}} \int_{A^r_{\varepsilon r}}
| G_{\rho}(y) - G_{\rho}(x) |^2 k(x,y) \dy \dx \right)^{1/2} \\
&\quad \times \left( \int_{A^{3r/2}_{\varepsilon r/2}}
\int_{A^r_{\varepsilon r}} | \eta(y) - \eta(x) |^2 k(x,y) \dy \dx \right)^{1/2}.
\end{align*}
By using the assumption \eqref{assum:U1}, we have
\begin{align*}
\int_{A^{3r/2}_{\varepsilon r/2}} \int_{A^r_{\varepsilon r}}
|\eta(x) - \eta(y)|^2 k(x,y) \dy \dx 
\leq 16r^{-2} \int_{A^{3r/2}_{\varepsilon r/2}} \int_{B(x, 5r/2)} 
|y-x|^2 k(x,y) \dy \dx 
\leq Cr^{d-\alpha},
\end{align*}
Therefore, it follows from \Cref{lem:local_term}, together with the estimate above, 
that
\begin{align} \label{eq:I_4_lower}
I_4 \leq C r^{d-\alpha} \sup_{A^{3r/2}_{\varepsilon r/2}} G_{\rho}.
\end{align}
Combining all estimates \eqref{eq:I_1,I_2,I_3,I_4_lower}--\eqref{eq:I_4_lower} 
we arrive at that
\begin{align*}
1 \leq C r^{d-\alpha} \sup_{A^{3r/2}_{\varepsilon r/2}} G_{\rho}(\cdot, y_0) + \frac{1}{2},
\end{align*}
or equivalently,
\begin{align*}
\sup_{A^{3r/2}_{\varepsilon r/2}} G_{\rho}(\cdot, y_0) \geq C | x_0 - y_0 |^{d-\alpha}.
\end{align*}
By taking $\rho \to 0$ and applying \Cref{lem:Harnack}, we have
\begin{align*}
\inf_{A^{3r/2}_{\varepsilon r/2}} G(\cdot, y_0) \geq C | x_0 - y_0 |^{\alpha-d}.
\end{align*}
Note that the above essential infimum is realized as the infimum since $G$ is 
continuous in $\Omega \setminus \lbrace y_0 \rbrace$.
Since $x_0 \in A^{3r/2}_{\varepsilon r/2}(y_0)$, we conclude the theorem.
\end{proof}

\section{Symmetry} \label{sec:symmetry}

This section is devoted to the proof of the symmetry of the Green function.

\begin{proof} [Proof of \Cref{thm:symm}]
Let $x, y \in \Omega$ with $x \neq y$. Recall that the Green function $G$ of $L$ on $\Omega$ is constructed by a sequence of regularized Green functions. Thus, we have a sequence $\rho_i < |x- y|/3$ converging to zero and a sequence of regularized Green functions $G_{\rho_i}( \cdot, y)$ converging to $G(\cdot, y)$. Let $\sigma_j < |x-y|/3$ be another sequence converging to zero and let $G_{\sigma_j}(\cdot, x)$ be the corresponding regularized Green functions converging to $G(\cdot, x)$. Then we have
\begin{align*}
a_{ij}:= \fint_{B_{\rho_i}(y)} G_{\sigma_j}(z, x) \,\d z = \mathcal{E}(G_{\rho_i}(\cdot, y), G_{\sigma_j}(\cdot, x)) = \fint_{B_{\sigma_j}(x)} G_{\rho_i}(z, y) \, \d z.
\end{align*}
Since $G_{\sigma_j}(\cdot, x)$ weakly converges to $G(\cdot, x)$ in $W^{\beta/2, q}_\Omega(\Rd)$ for all $\beta \in (0,\alpha)$ and $q \in [1, d/(d-\alpha/2))$, and $G(\cdot, y)$ is continuous on $B_{\sigma_j}(x)$, we have
\begin{align*}
\lim_{i \rightarrow \infty} \lim_{j \rightarrow \infty} a_{ij} = \lim_{i \rightarrow \infty} \fint_{B_{\rho_i}(y)} G(z, x) \, \d z = G(y, x).
\end{align*}
In the same way we obtain
\begin{align*}
\lim_{j \rightarrow \infty} \lim_{i \rightarrow \infty} a_{ij} = G(x, y).
\end{align*}
By \cite[Theorem 1.1]{KaWe22}, we have the uniform estimate on the H\"older norm of $G_{\rho_i}(\cdot, y)$ on $B_{\sigma_j}(x)$ which is independent of $i$. Therefore, the double sequence $a_{ij}$ converges uniformly in $j$ with respect to $i$. The proof is complete.
\end{proof}


\end{document}